\documentclass[12pt,a4paper,reqno]{amsart}
\usepackage{amssymb}
\usepackage{amscd}
\numberwithin{equation}{section}

\theoremstyle{plain}

\newtheorem{theorem}[subsection]{Theorem}
\newtheorem{proposition}[subsection]{Proposition}
\newtheorem{lemma}[subsection]{Lemma}
\newtheorem{corollary}[subsection]{Corollary}
\newtheorem{conjecture}[subsection]{Conjecture}

\theoremstyle{definition}

\newtheorem{definition}[subsection]{Definition}

\renewcommand{\leq}{\leqslant}
\renewcommand{\geq}{\geqslant}

\newsavebox{\proofbox}
\savebox{\proofbox}{\begin{picture}(7,7)%
  \put(0,0){\framebox(7,7){}}\end{picture}}

\newcommand{\md}[1]{\ensuremath{(\operatorname{mod}\, #1)}}
\newcommand{\mdsub}[1]{\ensuremath{(\mbox{\scriptsize mod}\, #1)}}

\newcommand\triv{{\operatorname{triv}}}
\newcommand\E{\mathbb{E}}
\newcommand\Z{\mathbb{Z}}
\newcommand\R{\mathbb{R}}

\newcommand\C{\mathbb{C}}

\newcommand\Factor{\mathcal{F}}
\renewcommand\P{\mathbb{P}}
\newcommand\F{\mathbb{F}}

\renewcommand\b{{\bf b}}
\newcommand\eps{\varepsilon}
\newcommand\rank{\operatorname{rank}}

\newcommand\Span{\operatorname{Span}}
\newcommand\obs{\operatorname{obs}}
\renewcommand\th{\operatorname{th}}
\newcommand\st{\operatorname{st}}
\def\proof{\noindent\textit{Proof. }}

\def\remark{\noindent\textit{Remark. }}
\def\remarks{\noindent\textit{Remarks. }}
\def\endproof{\hfill{\usebox{\proofbox}}}

\begin{document}

\title[Polynomials over finite fields and Gowers norms]{The distribution of polynomials over finite fields, with applications to the Gowers norms}

\author{Ben Green}
\address{Centre for Mathematical Sciences\\
Wilberforce Road\\
     Cambridge CB3 0WA\\
     England
}
\email{b.j.green@dpmms.cam.ac.uk}

\author{Terence Tao}
\address{UCLA Department of Mathematics, Los Angeles, CA 90095-1596.
}
\email{tao@math.ucla.edu}

\thanks{The first author is a Clay Research Fellow and gratefully acknowledges the support of the Clay Institute. The second author is supported by a grant from the MacArthur Foundation, and by NSF grant CCF-0649473.}


\begin{abstract}
In this paper we investigate the uniform distribution properties of polynomials in many variables and bounded degree over a fixed finite field $\F$ of prime order. Our main result is that a polynomial $P : \F^n \rightarrow \F$ is poorly-distributed only if $P$ is determined by the values of a few polynomials of lower degree, in which case we say that $P$ has \emph{small rank}.

We give several applications of this result, paying particular attention to consequences for the theory of the so-called \emph{Gowers norms}. We establish an inverse result for the Gowers $U^{d+1}$-norm of functions of the form $f(x)= e_\F(P(x))$, where $P : \F^n \rightarrow \F$ is a polynomial of degree less than $|\F|$, showing that this norm can only be large if $f$ correlates with $e_\F(Q(x))$ for some polynomial $Q : \F^n \rightarrow \F$ of degree at most $d$.

The requirement $\deg(P) < |\F|$ cannot be dropped entirely.  Indeed, we show the above claim fails in characteristic $2$ when $d = 3$ and $\deg(P)=4$, showing that the quartic symmetric polynomial $S_4$ in $\F_2^n$ has large Gowers $U^4$-norm but does not correlate strongly with any cubic polynomial. This shows that the theory of Gowers norms in low characteristic is not as simple as previously supposed.  This counterexample has also been discovered independently by Lovett, Meshulam, and Samorodnitsky \cite{lms}.

We conclude with sundry other applications of our main result, including a recurrence result and a certain type of nullstellensatz.
\end{abstract}

\maketitle

\section{Introduction}

Let $\F$ be a finite field of prime order.  Throughout this paper, $\F$ will be considered fixed (e.g. $\F = \F_2$ or $\F = \F_3$) and we shall be working inside the $n$-dimensional vector spaces $\F^n$ over $\F$ for various natural numbers $n$. More generally, any linear algebra term (e.g. span, independence, basis, subspace, linear transformation, etc.) will be understood to be over the field $\F$.

If $f: \F^n \to \C$ is a function, and $h \in \F^n$ is a shift, we define the (multiplicative) derivative $\Delta_h f: \F^n \to \C$ of $f$ by the formula
$$ \Delta_h f(x) := f(x+h) \overline{f(x)}.$$
An important special case arises when $f$ takes the form $f = e_\F(P)$, where $P: \F^n \to \F$ is a function, and $e_\F: \F \to \C$ is the standard character $e_\F(j) := e^{2\pi i j/|\F|}$ for $j=0,\ldots,|\F|-1$.  In that case we see that $\Delta_h f = e_\F( D_h P )$, where $D_h P: \F^n \to \F$ is the (additive) derivative of $P$, defined as
$$ D_h P(x) := P(x+h) - P(x).$$

Given an integer $d \geq 0$, we say that a function $P: \F^n \to \F$ is a \emph{polynomial of degree at most $d$} if we have $D_{h_1} \ldots D_{h_{d+1}} P = 0$ for all $h_1,\ldots,h_{d+1} \in \F^n$, and write $\mathcal{P}_d(\F^n)$ for the space of all polynomials on $\F^n$ of degree at most $d$. Thus for instance $\mathcal{P}_0(\F^n)$ is the space of constants, $\mathcal{P}_1(\F^n)$ is the space of linear polynomials on $\F^n$, $\mathcal{P}_2(\F^n)$ is the space of quadratic polynomials, and so forth.  It is easy to see that $\mathcal{P}_d(\F^n)$ is a vector space and that, with an obvious notation, the monomials $x_1^{i_1} \ldots x_n^{i_n}$ for $0 \leq i_1,\ldots, i_n < |\F|$ and $i_1+\ldots+i_n \leq d$ form a basis.  (The restriction $i_1,\ldots,i_n <|\F|$ arises of course from the fact that $x^{|\F|} = x$ for all $x \in \F$.)  We shall say that a function $f: \F^n \to \C$ is a \emph{polynomial phase of degree at most $d$} if it takes the form $f = e_\F(P)$ for some $P \in \mathcal{P}_d(\F^n)$, or equivalently if all $(d+1)^{\st}$ multiplicative derivatives $\Delta_{h_1} \ldots \Delta_{h_{d+1}} f$ are identically $1$.

It is of interest to test for the property that a function $P: \F^n \to \F$ is ``close to'' a polynomial of degree at most $d$, or to test for the closely related property that a function $f: \F^n \to \C$ ``correlates'' with a polynomial phase of degree at most $d$.  One proposal to perform such a test goes by the name of the \emph{Inverse Conjecture for the Gowers norms} (see e.g. \cite{bv,gt:inverse-u3,tao-stoc}), which roughly speaking asserts that a function $f$ correlates with a polynomial phase of degree at most $d$ if and only if the $(d+1)^{\st}$ multiplicative derivatives of $f$ are biased.  To describe this conjecture more precisely, we need some further notation.  

\begin{definition}[Gowers uniformity norm]\label{gd}\cite{gowers-4}, \cite{gowers}  Let $f: \F^n \to \C$ be a function, and let $d \geq 0$ be an integer.  We then define the \emph{Gowers norm} $\|f\|_{U^{d+1}}$ of $f$ to be the quantity\footnote{Here, as in all our papers, the expectation notation $\E_{x \in S}$ refers to the average $\frac{1}{|S|}\sum_{x \in S}$ over some finite non-empty set $S$. In this particular example, $S = (\F^n)^{d+1}$.}
$$ \|f\|_{U^{d+1}} := |\E_{h_1,\ldots,h_d,x \in \F^n} \Delta_{h_1} \ldots \Delta_{h_{d+1}} f(x)|^{1/2^{d+1}},$$
thus $\|f\|_{U^{d+1}}$ measures the average bias in $(d+1)^{\st}$ multiplicative derivatives of $f$.
We also define the \emph{weak Gowers norm} $\|f\|_{u^{d+1}}$ of $f$ to be the quantity
$$ \|f\|_{u^{d+1}} := \sup_{Q \in \mathcal{P}_d(\F^n)} |\E_{x \in \F^n} f(x) e_\F(-Q(x))|,$$
thus $\|f\|_{u^{d+1}}$ measures the extent to which $f$ can correlate with a polynomial phase of degree at most $d$.
\end{definition}

\begin{remark}  
It can in fact be shown that the Gowers and weak Gowers norm are in fact norms for $d \geq 2$ (and seminorms for $d=1$), see e.g. \cite{gowers,tao-vu}.  Further discussion of these two norms can be found in \cite{gt:inverse-u3}.  
\end{remark}

The Gowers norm and weak Gowers norm are closely related; for instance, one easily verifies the invariance 
\begin{equation}\label{polyphase}
\| f g \|_{U^{d+1}} =  \|f\|_{U^{d+1}} \hbox{ and }
\| f g \|_{u^{d+1}} =  \|f\|_{u^{d+1}}
\end{equation}
for all polynomial phases $g$ of degree at most $d$, and from this and the Cauchy-Schwarz-Gowers inequality (see e.g. \cite{tao-vu}) one can also verify the bound
\begin{equation}\label{uU}
 \| f\|_{u^{d+1}} \leq \|f\|_{U^{d+1}}
 \end{equation}
whenever $f$ is bounded in magnitude by $1$.  In the converse direction the following had been suggested, and was stated formally\footnote{The first-named author would like to make it clear that he also believed the conjecture.} in \cite{sam,tao-stoc}.

\begin{conjecture}[Inverse conjecture for the Gowers norm]\label{ig} Let $d \geq 0$, let $\delta \in (0,1]$, and $\F$ be a fixed finite field.  Suppose that $f : \F^n \rightarrow \C$ is a function with $|f(x)| \leq 1$ for all $x \in \F^n$ and for which $\Vert f \Vert_{U^{d+1}} \geq \delta$. Then
$\Vert f \Vert_{u^{d+1}} \gg_{d,\delta,\F} 1$; that is to say, there is some $c = c(d,\delta,\F) > 0$ such that $\Vert f \Vert_{u^{d+1}} \geq c$.
\end{conjecture}

This conjecture has been verified in a number of special cases.  For instance the case $d=0$ is trivial, and the case $d=1$ is easily established by Plancherel's theorem.  The case $d=2$ was established odd characteristic in \cite{gt:inverse-u3} and in the case $|\F|=2$ (which is of particular interest in theoretical computer science) in \cite{sam}.  The case when $\delta$ is sufficiently close to $1$ (depending on $d$ and $\F$) was established in \cite{akklr} (see also the earlier related work of \cite{blm} in the case $d=1$, and \cite{stv} in the case when $|\F|$ is assumed large compared to $d$ and $\delta$).

One of our results in this paper establishes a further special case of the conjecture, when the function $f$ is itself a polynomial phase, and the characteristic of $\F$ is not too small.

\begin{theorem}[Inverse conjecture for polynomial phases]\label{main-2} Suppose that $0 \leq d, k < |\F|$, and that $\delta \in (0,1]$. Let $P : \F^n \rightarrow \F$ be a polynomial of degree $k$, write $f(x) := e_{\F}(P(x))$, and suppose that $\| f \|_{U^{d+1}} \geq \delta$.  Then we have $\| f \|_{u^{d+1}} \gg_{\F,\delta} 1$.  
\end{theorem}

Note carefully the lower bound on the characteristic $|\F|$ of $\F$. It turns out that some such restriction is necessary, and indeed that Conjecture \ref{ig} is false without some modification. This is elucidated by the following example, which we shall analyse in \S \ref{counter-sec}.  For any $d \geq 0$ and any vector space $\F^n$, let $S_d \in \mathcal{P}_d(\F^n)$
be the symmetric polynomial of degree $d$:
\begin{equation}\label{pdef}
 S_d(x_1,\ldots,x_n) := \sum_{1 \leq i_1 < \ldots < i_d \leq n} x_{i_1} \ldots x_{i_d}.
\end{equation}

\begin{theorem}[Counterexample for the $U^4$-norm in $\F_2$]\label{main-1}  Let $n$ be a large integer.  Then the function $f: \F_2^n \to \{-1,1\}$ defined by $f := e_{\F_2}(S_4) = (-1)^{S_4}$ is such that
\begin{equation}\label{u4}
\|f\|_{U^4}^{16} = \frac{1}{8} + O( 2^{-n/2} )
\end{equation}
but such that
\begin{equation}\label{u4b}
 \|f\|_{u^4} \ll (\log n)^{-c}
\end{equation}
for some absolute constant $c > 0$.
\end{theorem}

This counterexample was discovered independently by Lovett, Meshulam and Samorodnitsky \cite{lms}. They obtain a very much stronger bound for the lack of correlation of $f$ with a cubic phase, namely $\Vert f \Vert_{u^4} \ll 2^{-cn}$. We obtain our bound by a very slight modification of Ramsey-theoretic arguments of Alon and Beigel \cite{beigel}.  We will in fact be able to establish similar results with $S_4$ replaced by $S_{2^j}$ for $j \geq 2$; see Theorem \ref{sdgow}. The aforementioned paper of Lovett, Meshulam and Samorodnitsky goes further in establishing counterexamples to Conjecture \ref{ig} for all prime fields $\F = \F_p$; specifically, the conjecture fails when $d +1 = p^2$.

We note that the counterexample presented in Theorem \ref{main-1} is also a counterexample to the specific case of Conjecture \ref{ig} given as \cite[Conjecture 21]{bv}.

It seems of interest to determine for what other degrees, Gowers norms, and characteristics one has a counterexample of the above type, and to ask what can be salvaged when $\F$ is very small. We will speculate on these questions in \S \ref{gen-sec}. We do not regard Theorem \ref{main-1} as an obstacle to the possible truth of the inverse conjecture over $\Z/N\Z$ on which our programme to count solutions to linear equations in primes depends (cf. \cite{green-tao-linearprimes}). Indeed this seems to be a ``low characteristic'' issue, albeit one of a rather interesting nature.

We turn now to a discussion of the main technical result of the paper, on which the proof of Theorem \ref{main-2} depends. We begin by defining the notion of \emph{rank}.

\begin{definition}[Rank]\label{rank-def}  Let $d \geq 0$, and let $P: \F^n \to \F$ be a function.  We define the \emph{degree $d$ rank} $\rank_d(P)$ of $P$ to be the least integer $k \geq 0$ for which there exist polynomials $Q_1,\ldots,Q_k \in \mathcal{P}_d(\F^n)$ and a function $B: \F^k \to \F$ such that we have the representation $P = B(Q_1,\ldots,Q_k)$.  If no such $k$ exists, we declare $\rank_d(P)$ to be infinite (since $\F^n$ is finite-dimensional, this only occurs when $d=0$ and $P$ is non-constant).
\end{definition}

In the low-degree case, it is well known that the bias $\E_{x \in \F^n} e_\F(P(x))$ of a polynomial phase $e_\F(P(x))$ is closely related to the rank of $P$.  For instance, if $P \in \mathcal{P}_1(\F^n)$ is linear, then from simple Fourier analysis we see that $\E_{x \in \F^n} e_\F(P(x))$ has magnitude $1$ if $\rank_0(P) = 0$ and magnitude $0$ otherwise.  For quadratic polynomials, we have the following well-known fact:

\begin{lemma}[Gauss sum estimate]\label{gauss}  If $P \in \mathcal{P}_2(\F^n)$, then
$$ |\E_{x \in \F^n} e_\F(P(x))| \ll |\F|^{-c\rank_1(P)}$$
where $c > 0$ is an absolute constant.
\end{lemma}

\proof   If $P \in \mathcal{P}_1(\F^n)$ then the claim can be verified by Fourier analysis, so we can assume that $P \not \in \mathcal{P}_1(\F^n)$.  We begin with the easy case $|\F| > 2$, and then discuss the changes needed to handle $|\F| = 2$.

Suppose that 
\begin{equation}\label{evpd}
|\E_{x \in \F^n} e_\F(P(x))| \geq \delta
\end{equation}
for some $0 < \delta < 1/2$.  It will suffice to show that $\rank_1(P) \ll \log_{|\F|} \frac{1}{\delta}$. 

Squaring \eqref{evpd}, we conclude that
$$ \delta^2 \leq \E_{x,y \in \F^n} e_\F(P(x)-P(y)) = \E_{x,h \in \F^n}  e_\F(D_h P(x)).$$
From Fourier analysis, we see that the average $\E_{x \in \F^n} e_\F(D_h P(x))$ vanishes unless $D_h P \in \mathcal{P}_0(\F^n)$, in which case it has magnitude $1$.  Thus the assumption \eqref{evpd} implies that
$$ \P_{h \in \F^n}( D_h P \in \mathcal{P}_0(\F^n) ) \geq \delta^2.$$
Now by breaking up $P$ into monomials, we can express $P(x) = B(x,x) + L(x)$ for some bilinear form $B: \F^n \times \F^n \to \F$ and some $L \in \mathcal{P}_1(\F^n)$.  In the odd characteristic case $|\F| > 2$, we can take $B$ to be symmetric.  We conclude that
$$ D_h P(x) = 2 B(x,h) \md{\mathcal{P}_0(\F^n)},$$
and hence that
$$ \P_{h \in \F^n}( B(x,h) = 0 \hbox{ for all } x \in \F^n ) \geq \delta^2.$$
If $\delta^2 > 1/|\F|$ then this forces $B$ to vanish identically, which contradicts the hypothesis $P \not \in \mathcal{P}_1(\F^n)$, so we may assume $\delta^2 \leq 1/|\F|$.  Then the linear transformation associated to $B$ has rank at most $O( \log_{|\F|} 1/\delta )$; since $P(x) = B(x,x) + L(x)$, we conclude $\rank_1(P) \ll \log_{|\F|} 1/\delta$ as desired.

Now we consider the even characteristic case $|\F|=2$, in which case we cannot take $B$ to be symmetric.  Then the above argument gives 
$$ \P_{h \in \F^n}( \tilde B(x,h) = 0 \hbox{ for all } x \in \F^n ) \geq \delta^2.$$
where $\tilde B(x,h) := B(x,h) + B(h,x)$ is a symmetric bilinear form.   Thus $\tilde B$ must have rank $O( \log_2 1/\delta )$.   By linear algebra we can thus express
$$ \tilde B(x,h) = \sum_{1 \leq i,j \leq k} c_{i,j} L_i(x) L_j(h)$$
for some $k \ll \log_2 1/\delta$, some linearly independent linear functionals $L_i: \F^n \to \F$, and some coefficients $c_{i,j} \in \F$.  Since $\tilde B$ is symmetric and the $L_i$ are independent, we have $c_{i,j} = c_{j,i}$.  Since $\tilde B(x,x) = B(x,x) + B(x,x)$ vanishes in characteristic $2$, we also see that $c_{i,i}=0$.  We can thus write
$$ \tilde B(x,h) = C(x,h) + C(h,x)$$
where $C(x,h) := \sum_{1 \leq i < j \leq k} c_{i,j} L_i(x) L_j(h)$ is the lower-triangular component of $\tilde B(x,h)$.  We then easily verify that $B(x,x)-C(x,x)$ is a linear function of $x$, and so $P(x)$ can be expressed as the sum of $C(x,x)$ and a linear function, from which the claim $\rank_1(\P) \ll \log_2 1/\delta$ follows.
\endproof

We shall establish the following generalisation of the above estimate to higher degree polynomials, provided that the degree does not exceed the characteristic:

\begin{theorem}[Lack of equidistribution implies bounded rank]\label{wtn}  Suppose that an integer $d$ satisfies $0 \leq d < |\F|$. Let $\delta \in (0,1]$, and suppose that $P \in \mathcal{P}_d(\F^n)$ is such that $|\E_{x \in \F^n} e_\F(P(x))| \geq \delta.$
Then $\rank_{d-1}(\P) \ll_{\F, \delta, d} 1.$
\end{theorem}

The proof of this theorem is the technical heart of the paper, and will be accomplished in \S \ref{proof-sec}. It is possible that the restriction on $|\F|$ can be removed, but our method of proof breaks down when $d \geq |\F|$. Certainly the deduction of Theorem \ref{main-2} from Theorem \ref{wtn} breaks down in this case (which of course it must, thanks to Theorem \ref{main-1}).

\emph{Acknowledgements.} The authors are indebted to Andrej Bogdanov, Tali Kaufman, and Emanuele Viola for suggesting this problem, and for many useful discussions.  The authors also thank Alex Samorodnitsky for drawing attention to the recent preprint \cite{lms}, and to Peter Sarnak for suggestions.

\section{Factors and regularity}
In this section we give some definitions and results which will be useful in our proof of Theorem \ref{wtn}.

\begin{definition}[Factors and configuration space]\label{factor-def} Suppose that $d \geq 0$ is an integer and that $M_1,\dots,M_d$ are further non-negative integers. By a \emph{factor} of degree $d$ on $\F^n$ we mean a collection $\Factor = (P_{i,j})_{1 \leq i\leq d, 1 \leq j \leq M_i}$ where $P_{i,j} \in \mathcal{P}_i(\F^n)$ for all $i,j$. By the \emph{dimension} $\dim(\Factor)$ of $\Factor$ we mean the quantity $M_1 + \dots + M_d$. Write $\Factor_i$ for the $i$-degree part of $\Factor$, that is to say the collection $(P_{i,j})_{1 \leq j \leq M_i}$. Although we are using the term factor to describe nothing more complicated than a collection of polynomials, we encourage the reader to think in addition of the $\sigma$-algebra $\sigma(\Factor)$ defined by these polynomials $P_{i,j}$, that is to say the partition of $\F^n$ into \emph{atoms} of the form $\{x : P_{i,j}(x) = c_{i,j}\}$. 
We write $\Sigma = \F^{M_1} \times \dots \F^{M_d}$ and call this the \emph{configuration space} of $\Factor$. We write $\Phi : \F^n \rightarrow \Sigma$ for the \emph{evaluation map} given by $\Phi(x) = (P_{i,j}(x))_{1 \leq i \leq d, 1 \leq j \leq M_i}$.
\end{definition}

We will use the notation of this definition throughout the paper without further comment. Sometimes we will have factors $\Factor,\Factor'$ and $\Factor''$; we will write $P_{i,j},P'_{i,j},P''_{i,j}$, $\Sigma, \Sigma', \Sigma''$, $M_j, M'_j, M''_j$, $\Phi, \Phi', \Phi''$ and so on for the corresponding polynomials, configuration spaces, dimensions and evaluation maps.

We will frequently need to \emph{extend} a factor into a more \emph{regular} one, by expressing the complicated polynomials in a factor by simpler ones.  Our notation for this concept is as follows. We say that a factor $\Factor'$ is an \emph{extension} of $\Factor$ if $\sigma(\Factor')$ is a (possibly trivial) refinement of $\sigma(\Factor)$. Note that this is \emph{not} the same thing as saying that the collection $(P'_{i,j})$ defining $\Factor'$ contains the collection $(P_{i,j})$ defining $\Factor$. For example, the factor defined by the linear polynomials $x_1,x_2,x_3$ is a refinement of that defined by the polynomials $x_1,x_2$ and $x_1 + x_2$.

By a \emph{growth function of order $d$} we mean a non-decreasing function $F: \Z^+ \to \R^+$.  

\begin{definition}[$F$-regularity]  Let $\Factor$ be a factor of order $d$, and let $F$ be a growth function.  We say that $\Factor$ is \emph{$F$-regular} if we have
$$ \rank_{i-1}(\sum_{j=1}^{M_i} c_{i,j} P_{i,j}) \geq F(\dim(\Factor))$$
for all $1 \leq i \leq d$ and all coefficients $c_{i,1},\ldots,c_{i,M_i} \in \F$ that are not all zero.  (In particular, if $F$ is positive, this implies that the polynomials $P_{i,1},\ldots,P_{i,M_i}$ are linearly independent.)
\end{definition}

\begin{example} If $d$, $F$ and $M_1,\ldots,M_d$ are fixed, and $P_{i,j}$ are chosen uniformly at random from $\mathcal{P}_i(\F^n)$, then the resulting factor $\Factor$ will be $F$-regular with probability $1-o(1)$, where $o(1)$ goes to zero as $n \to \infty$ for fixed $d,F,M_1,\ldots,M_d$.  Indeed, one should view the polynomials in an $F$-regular factor as ``behaving like'' generic polynomials, in that they obey no unexpected algebraic constraints of bounded complexity.
\end{example}

The following lemma, which allows us to replace take an arbitrary factor $\Factor$ and find a highly regular extension of it, is absolutely fundamental to our arguments. This generalises \cite[Lemma 8.7]{green-tao-finfieldAP4s} to the case of factors of degree 3 or more. The result is faintly analagous in some ways to \emph{Szemer\'edi's regularity lemma} for graphs and to more recent versions of this for hypergraphs.

\begin{lemma}[Regularity lemma]\label{regular} Let $d \geq 1$, let $F$ be a growth function, and let $\Factor$ be a factor of degree $d$.  Then there exists an $F$-regular extension $\Factor'$ of $\Factor$ of degree $d$ satisfying the dimension bound
$$ \dim(\Factor') \ll_{F,d,\dim(\Factor)} 1.$$
\end{lemma}
\remark The actual bound we obtain here, if one worked it out, would have an extremely weak dependence on $F,d$ and $\dim(\Factor)$. Even for quite ``reasonable'' growth functions $F$ one starts to see functions in the Ackerman hierarchy making an appearance. It is our dependence on this lemma and the rather poor bounds that result from its proof that renders Theorem \ref{wtn} essentially ineffective.

\proof   Fix $d$ and $F$.  We shall induct on the dimension vector $(M_1,\dots,M_d)$ of $\Factor$ where, of course, $M_i := \dim(\Factor_i)$.  This dimension vector takes values in $\Z_+^d$, which we shall order in reverse lexicographical ordering, that is to say $(M_1,\ldots,M_d) < (M'_1,\ldots,M'_d)$ if there exists $1 \leq i \leq d$ such that $M_i < M'_i$ and $M_j = M'_j$ for all $i < j \leq d$.  This turns $\Z_+^d$ into a well-ordered set (with the ordinal type $\omega^d$), and so we can perform strong induction on this space.  In other words, we may assume without loss of generality that the claim has already been proven for all smaller dimension vectors.

If $\Factor$ is already $F$-regular, then we are done.  Otherwise, there exists $i \in [d]$ and a non-trivial linear combination $Q_i$ of the $P_{i,1},\ldots,P_{i,M_i}$ such that $\rank_{i-1}(Q_i) < F(\dim(\Factor))$, or in other words $Q_i$ is some combination of fewer than $F(\dim(\Factor))$ polynomials of degree at most $i-1$.  By rewriting $Q_i$ in this fashion, we can find an extension $\Factor''$ of $\Factor$ with dimension vector 
$$(M_1,\ldots,M_{i-1} + \lfloor F(\dim(\Factor)) \rfloor, M_i-1, M_{i+1}, \ldots, M_d)$$
(with some obvious modifications in the easy case $i=1$).  Applying the induction hypothesis to $\Factor''$ we obtain the claim.
\endproof

\section{A lemma of Bogdanov and Viola}

In this section we recall \cite[Lemma 25]{bv}, and provide a proof in the interests of self-containment. This lemma \emph{almost} immediately establishes our main result, Theorem \ref{wtn}, except for the presence of some small errors. Our main task in subsequent sections is to eliminate the errors and turn this near-miss result into a proof of Theorem \ref{wtn}.  

\begin{lemma}[Bogdanov-Viola lemma]\label{bog-viola-lem}
Let $d \geq 0$ be an integer, and let $\delta,\sigma \in (0,1]$ be parameters. Suppose that $P \in \mathcal{P}_d(\F^n)$ is a polynomial of degree $d$ such that \begin{equation}\label{assump}|\E_{x \in \F^n} e_{\F}(P(x))| \geq \delta.\end{equation} Then there exists a function $\tilde P : \F^n \rightarrow \F$ with $\rank_{d-1} (\tilde P) \leq \frac{|\F|^5}{\delta^2\sigma}$ such that $\P_{x \in \F^n}(P(x) \neq \tilde P(x)) \leq \sigma$.
\end{lemma}
\proof We remark that the bound on $\rank_{d-1}(\tilde P)$ is \emph{much} superior to that we will eventually obtain for Theorem \ref{wtn}. This is because the Bogdanov-Viola lemma does not rely on the regularity lemma, Lemma \ref{regular}. In fact this bound could even be improved somewhat, but this is not relevent to our work here. 

For each $r \in \F$, define a measure $\mu_r : \F \rightarrow [0,1]$ by setting
\[ \mu_r(t) = \P_{x \in \F^n}(P(x) = t+r)\] for all $t \in \F$. Then \eqref{assump} implies that $|\sum_{t \in \F} e_{\F}(t)\mu_0(t)| \geq \delta$.
Noting that
\[ \sum_{t \in \F} e_\F(t)\mu_0(t) = e_\F(d) \sum_t e_\F(t) \mu_d(t),\]
we see that
\[ \Vert \mu_0 - \mu_{d}\Vert  := \sum_t |\mu_0(t) - \mu_d(t)|  \geq |1 - e_{\F}(d)| |\sum_t e_\F(t) \mu_0(t)| \geq 4\delta/|\F|
\]
if $d \neq 0$, by dint of the inequality $|1 - e^{2\pi i \theta}| \geq 4|\theta|$ which holds when $|\theta| \leq 1/2$.
By translation invariance we conclude that 
\begin{equation}\label{measure-separation} \Vert \mu_r - \mu_s \Vert \geq 4\delta/|\F|\end{equation} whenever $r \neq s$.

Now fix a value of $x$ and let $h \in \F^n$ be chosen at random. Then 
\[ \P_h(D_hP(x) = t) = \P_h(P(x+h) = t + P(x)) = \mu_{P(x)}(t),\]
that is to say $D_hP(x)$ has the distribution $\mu_{P(x)}$. Now we expect that if a large number $D_{h_1}P(x),\dots,D_{h_k}P(x)$ of points are sampled from this distribution then the observed distribution
\[ \mu_{\obs}(h_1,\dots,h_k;x) := \frac{1}{k}\sum_{i = 1}^k \delta_{D_{h_i}P(x)}\] should approximate $\mu_{P(x)}$. In view of the separation property \eqref{measure-separation}, this ought to give us a good chance of recovering $P(x)$. 

Choose $k \geq \frac{|\F|^5}{2\sigma\delta^2}$, and sample $h_1,\dots,h_k$ independently at random from $\F^n$. Motivated by the above discussion, we define $\tilde P_{h_1,\dots,h_k}(x)$ to be that value of $r \in \F$ for which $\Vert \mu_{\obs}(h_1,\dots,h_k;x) - \mu_r\Vert$ is minimal. Note that $\tilde P_{h_1,\dots,h_k}$ is measurable with respect to the set of functions $D_{h_1}P(x),\dots,D_{h_k}P(x)$, each of which is a polynomial of degree at most $d-1$. Thus
\[ \rank_{d-1} (\tilde P_{h_1,\dots,h_k}) \leq k.\]
It remains to show that, at least for some choice of $h_1,\dots,h_k$, the function $\tilde P_{h_1,\dots,h_k}$ approximates $P$. Now if $\tilde P_{h_1,\dots,h_k}(x) \neq P(x)$ then it follows from the separation property \eqref{measure-separation} that 
\[ \Vert \mu_{\obs}(h_1,\dots,h_k,x) - \mu_{P(x)}\Vert \geq 2\delta/|\F|.\]
We claim that for fixed $x$ the probability of this happening (over random choices of $h_1,\dots,h_k$) is at most $\sigma$. Summing over $x$, it then follows that there is at least one choice of $h_1,\dots,h_k$ for which $\# \{x : P(x) \neq \tilde P_{h_1,\dots,h_k}(x)\} \leq \sigma |\F^n|$, and the lemma follows upon taking $\tilde P := \tilde P_{h_1,\dots,h_k}$.

Fix $x \in \F^n$ and a value of $t \in \F$, and write $Y_i = 1_{D_{h_i}P(x) = t}$. To establish the claim, it suffices to show that
\[ \P\big( |\frac{Y_1 + \dots + Y_k}{k} - \mu_{P(x)}(t)| \geq \frac{2\delta}{|\F|} \big) \leq \frac{\sigma}{|\F|}.\] 
Noting that the $Y_i$ are i.i.d. Bernouilli random variables with means $\overline{Y} = \mu_{P(x)}(t)$, this follows from a suitable version of the law of large numbers. In this case we may use the inequality
\[ \P\big(| \frac{Y_1 + \dots + Y_k}{k} - \overline{Y} | \geq \eta \big) \leq \frac{1}{4k\eta^2},\] which follows from Chebyshev's inequality.\endproof

\remark When $|\F| = 2$, the above proof has a pleasant interpretation. The value of $\tilde P_{h_1,\dots,h_k}(x)$ is then obtained by ``majority vote'' amongst the values of $D_{h_i}P(x)$.

\section{Counting lemmas}

We shall prove Theorem \ref{wtn} by induction.  Accordingly, we begin by first describing some \emph{consequences} of Theorem \ref{wtn} at a given order $d$, which are already of some independent interest.  These consequences complement the regularity lemma in much the same way that ``counting lemmas'' in graph theory complement the Szemer\'edi regularity lemma. 

\begin{lemma}[Size of atoms]\label{atom-size} Let $d \geq 1$, and $\eps > 0$.  Suppose that Theorem \ref{wtn} is true for orders up to $d$.  Then there exists a growth function $F$ \textup{(}depending on $d$ and $\eps$\textup{)} such that if $\mathcal{F}$ is an $F$-regular factor of order $d$ on $\F^n$ then we have the estimate
\begin{equation}\label{pv}
\P_{x \in \F^n}(\Phi(x) = t ) = (1 + O(\eps)) \frac{1}{|\Sigma|}
\end{equation}
for all configurations $t \in \Sigma$. In words, all the atoms in the $\sigma$-algebra $\sigma(\Factor)$ have roughly the same size.
\end{lemma}
\remark Recall that $\Sigma = \F^{M_1} \times \dots \times \F^{M_d}$ is the configuration space associaed to the factor $\Factor$, and that $\Phi : \F^n \rightarrow \Sigma$ is the evaluation map.

\proof We may expand the condition $\Phi(x) = t$ using Fourier analysis on $\Sigma$ to obtain
\[ \P_x(\Phi(x) = t) = \frac{1}{|\Sigma|} \sum_{r \in \Sigma} \E_{x \in \F^n} e_{\F}( r \cdot (\Phi(x) - t)).\]
It therefore suffices to show that
\begin{equation}\label{Evf}
\E_{x \in \F^n} e_\F( \sum_{i=1}^d Q_i ) = O( \eps / |\F|^{\dim(\Sigma)} )
\end{equation}
whenever the $Q_i \in \Span(\Factor_i)$ are not all zero. Let $s \in [d]$ be the largest integer for which $Q_s$ is non-zero.  
As $\Factor$ is $F$-regular, we have $\rank_{s-1}(Q_s) \geq F(\dim(\Factor))$.  On the other hand, $\sum_{i=1}^d Q_i$ differs from $Q_s$ by an element of $\mathcal{P}_{s-1}(V)$.  Thus
$$ \rank_{s-1}(\sum_{i=1}^d Q_i) \geq F(\dim(\Factor)) - 1.$$
If we choose $F$ to sufficiently rapidly growing depending on $\eps$ and $d$, we can thus invoke Theorem \ref{wtn} to obtain \eqref{Evf} as required.  
\endproof

In addition to understanding the distribution of $\Phi(x)$, it turns out to be important to have an understanding of how $k$-dimensional \emph{parallelepipeds} are distributed in configuration space. That is, we study the distribution of $(\Phi(x + \omega \cdot h))_{\omega \in \{0,1\}^k}$ in $\Sigma^{\{0,1\}^k}$, where $h = (h_1,\dots,h_k)$ is a $k$-tuple of elements of $\F^n$. When $k = 2$, for example, we are interested in the $4$-tuple $(\Phi(x),\Phi(x+h_1),\Phi(x+h_2),\Phi(x+h_1 + h_2))$. We prepare the ground for this study with some definitions.

\begin{definition}[Faces and lower faces]
Let $k \geq 1$ be an integer and suppose that $0 \leq k' \leq k$. A subset $F \subseteq \{0,1\}^k$ is called a \emph{face} of dimension $k'$ if it has the form
\[ F = \{\omega \in \{0,1\}^k : \omega_i = \delta_i \quad \mbox{for $i \in I$}\},\]
where $I \subseteq [k]$ has size $k-k'$ and each $\delta_i$ is either 0 or 1. If all of the $\delta_i$ are zero then we say that $F$ is a \emph{lower face}. A lower face of dimension $k'$ can be identified with the power set of $[k]\setminus I$, which is a set of size $k'$. 
\end{definition}

Suppose that we have a parallelepiped $(x + \omega \cdot h)_{\omega \in \{0,1\}^k}$ in $\F^n$, where $h = (h_1,\dots,h_k)$ is a $k$-tuple of elements of $\F^n$. Consider the image
$(\Phi(x + \omega \cdot h))_{\omega \in \{0,1\}^k} \in \Sigma^{\{0,1\}^k}$. This cannot be arbitrary: indeed we have the ``obvious'' constraints coming from the relations
\[ \sum_{\omega \in F} (-1)^{|\omega|} P_{i,j}(x + \omega \cdot h) = 0\]
whenever $F \subseteq \{0,1\}^k$ is a face of dimension at least $i+1$, and $|\omega| := \omega_1+\ldots+\omega_k$.   To model these obvious constraints, we introduce some more notation.

\begin{definition}[Face vectors and parallelepiped constraints]
Suppose that $i_0 \in [d]$, that $j_0 \in [M_{i_0}]$ and that $F \subseteq \{0,1\}^k$. Consider the vector $r(i_0,j_0,F) \in \Sigma^{\{0,1\}^k}$ for which $r_{i,j}(\omega) = (-1)^{|\omega|}$ if $i = i_0$, $j = j_0$ and $\omega \in F$, and is zero otherwise. We call such a vector a \emph{face vector}. If $F$ is a lower face then we speak of a \emph{lower face vector}. If $\dim(F) \geq i_0 + 1$ we say that the face vector (or lower face vector) is \emph{relevant}. We say that $(t(\omega))_{\omega \in \{0,1\}^k} \in \Sigma^{\{0,1\}^k}$ \emph{satisfies the parallelepiped constraints} if it is orthogonal to all the relevant lower face vectors.
\end{definition}

\remarks The motivation for this definition, of course, is that for any $x, h_1,\dots,h_k$ the vector $(\Phi(x + \omega \cdot h))_{\omega \in \{0,1\}^{k}} \in \Sigma^{\{0,1\}^k}$ satisfies the parallelepiped constraints. At first sight the fact that we have restricted attention to \emph{lower} face vectors may look curious. However it turns out (and is not hard to prove) that the set of relevant face vectors in $\Sigma^{\{0,1\}^k}$ is spanned by the relevant lower face vectors. We will not require this fact. 

Write $\Sigma_{\Box} \subseteq \Sigma^{\{0,1\}^k}$ for the subspace of vectors in $\Sigma^{\{0,1\}^k}$ satisfying the parallelepiped constraints.

\begin{lemma}[Dimension of $\Sigma_{\Box}$]\label{dimlem}
Suppose that $k > d$. Then we have
$$\dim(\Sigma_{\Box}) = \sum_{i=1}^d M_i \sum_{0 \leq j \leq i} \binom{k}{j}.$$
\end{lemma}
\proof Since $\dim(\Sigma^{\{0,1\}^k}) = 2^k(M_1 + \dots + M_d) = \sum_{i=1}^d M_i \sum_j \binom{k}{j}$, it suffices to show that the dimension of the space spanned by the relevant lower face vectors is $\sum_{i=1}^d M_i \sum_{j > i} \binom{k}{j}$. This is precisely the number of different relevant lower face vectors, and so we must only show that the lower face vectors are linearly independent. To do this, we may clearly work with a fixed choice of $i$ and $j$, since the supports of the face vectors $r(i,j,F)$ are disjoint for different pairs $(i,j)$. Suppose there is some linear relation
\[ \sum_F a_F r(i,j,F) = 0.\]
Among all lower faces $F$ for which $a_F \neq 0$, suppose that $F_0$ contains the largest element $\omega_0$ in the lexicographic order on $\{0,1\}^k$. Comparing coefficients of $\omega_0$ we see that $a_{F_0} = 0$, contrary to assumption.\endproof 
 
If the factor $\Factor$ is $F$-regular for some sufficiently rapid growth function $F$, it turns out that the parallelepiped constraints we have written down are the only relevant ones in a rather strong sense. 

\begin{proposition}[Counting parallelepipeds]\label{parallel-count}
Suppose that $|\F|, k > d$, and suppose that Theorem \ref{wtn} is true for orders up to $d$. Let $\eps \in (0,1)$ be a parameter and suppose that $F$ grows sufficiently quickly \textup{(}depending on $k,d$ and $\eps$\textup{)}. Suppose that the factor $\Factor$ has degree at most $d$ and is $F$-regular. Suppose that $t_{\Box} \in \Sigma_{\Box}$, and that $x \in \F^n$ is a point with $\Phi(x) = t_{\Box}(0)$. Then the number of $h \in (\F^n)^k$ such that $\Phi(x + \omega \cdot h) = t_{\Box}(\omega)$ for all $\omega \in \{0,1\}^k$ is $1 + O_k(\eps)$ times $|\F|$ to the power $nk -  \sum_{i=1}^d M_i \sum_{1 \leq j \leq i} \binom{k}{j}$.
\end{proposition}
\remark Note carefully that we have been able to fix the basepoint $x$; this is important in applications of the proposition.  This is why $j$ now only ranges from $1$ to $i$ rather than from $0$ to $i$ as in Lemma \ref{dimlem}.

\proof Write $\Phi_{\Box}(h)$ for the vector $(\Phi(x + \omega \cdot h))_{\omega \in \{0,1\}^k}$ in $\Sigma^{\{0,1\}^k}$. We seek the number of $h$ for which $\Phi_{\Box}(h) = t_{\Box}$; by harmonic analysis on $\Sigma^{\{0,1\}^k}$ this may be expanded as

\begin{equation}\label{eq1} |\F|^{nk}|\Sigma^{\{0,1\}^k}|^{-1}\sum_{r_{\Box} \in \Sigma^{\{0,1\}^k}}\E_{h \in (\F^n)^k} e_{\F}(r_{\Box} \cdot (\Phi_{\Box}(h) - t_{\Box})).\end{equation}

Now when $r_{\Box}$ lies in the space $W$ spanned by the relevant lower face vectors together with the vectors $r(i,j,0)$ we have $r_{\Box} \cdot (\Phi_{\Box}(h) - t_{\Box}) = 0$, since both $\Phi_{\Box}(h)$ and $t_{\Box}$ satisfy the parallelepiped constraints and $\Phi_{\Box}(h)(0) = t_{\Box}(0)$. Since the lower face vectors are linearly independent the contribution from these $r_{\Box}$ to the sum \eqref{eq1} is $|\F|$ to the power $nk - \sum_{i=1}^d M_i \sum_{1 \leq j \leq i} \binom{k}{j}$. To conclude the argument it certainly suffices to show that the contribution from each $r_{\Box} \notin W$ is small in the sense that 
\begin{equation}\label{eq2}
|\E_{h \in (\F^n)^k} e_{\F}(r_{\Box} \cdot \Phi_{\Box}(h))| \leq \eps |\F|^{-2^k \dim(\Sigma)}.
\end{equation}
Such an exponential sum is unaltered in magnitude if an arbitrary element of $W$ is added to $r_{\Box}$. By repeated operations of this type, directed so as to reduce the largest element in the $\omega$-support of each $(r_{\Box}(\omega))_{i,j}$ in the lexicographic order on $\{0,1\}^k$, we may assume that $(r_{\Box}(\omega))_{i,j} = 0$ unless $|\omega| \leq i$. Since $r_{\Box}$ is not in $W$, there is at least one choice of $i,j$ and at least one $\omega \neq 0$ for which $(r_{\Box}(\omega))_{i,j} \neq 0$. Amongst all such triples $(i,j,\omega)$, choose one with the largest value of $i$, say $i = i_0$. For this value of $i = i_0$ choose $(j_0,\omega_0)$ with $s = |\omega_0|$ maximal, still subject to the condition that $(r_{\Box}(\omega_0))_{i_0,j_0} \neq 0$. Note that $1 \leq s \leq i$. By relabelling the cube $\{0,1\}^k$ we may assume that $\omega_0 = 1^s 0^{k-s}$. By construction, any triple $(i,j,\omega)$ satisfies one of the following properties:

\begin{enumerate}
\item $i > i_0$ and $\omega = 0$;
\item $i = i_0$ and $\omega = \omega_0$;
\item $i = i_0$ and at least one of the coordinates $\omega_l$, $1 \leq l \leq s$, is zero;
\item $i < i_0$.
\end{enumerate}
Since $1 \leq s \leq i \leq  k$ the sum in \eqref{eq2} may then be written as an average (over $h_{s+1},\dots,h_k$) of sums of the form
\[ \E_{h_1,\dots,h_s} e_{\F}(P(x + h_1 + \dots + h_s) + Q(h_1,\dots,h_s)),\] where $P$ is not zero and lies in $\Span(\Factor_i)$, and $Q$ has degree at most $s-1$ as a polynomial in $h_1,\dots,h_s$. Such a sum may be written as
\[ \E_{h_1,\dots,h_s} \b_1(h_2,\dots,h_s) \dots \b_s(h_1,\dots,h_{s-1})e_{\F}(P(x + h_1 + \dots + h_s)),\]
where each $\b$ is a bounded function which does not depend on $h_i$.
By introducing dummy variables we may assume that $s = i$. Applying the Cauchy-Schwarz inequality $i$ times to eliminate the bounded functions $\b$, we see that the sum in \eqref{eq2} may be bounded thus:
\[ |\E_{h \in (\F^k)^n}e_{\F}(r_{\Box} \cdot \Phi_{\Box}(h))| \leq  \big( \E_{h_1,\dots,h_i} e_{\F}(D_{h_1}\dots D_{h_i}P( \cdot )) \big)^{1/2^i}.\]
Note that this derivative is, for fixed $h_1,\dots,h_i$, simply a constant; we write it as $\partial^i P(h_1,\dots,h_i)$. It follows that if \eqref{eq2} is false then
\[ |\E_{h_1,\dots,h_i} e_{\F}(\partial^i P(h_1,\dots,h_i))| \geq (\eps |\F|^{-2^k \dim(\Sigma)})^{2^i}.\]
Applying Theorem \ref{wtn} at degree $i \leq d$ and with $V = (\F^n)^i$ we see that 
\[ \rank_{i-1}(\partial^i P) \ll_{k,\eps,\dim(\Sigma)} 1.\]
Note however that we have the Taylor expansion
\[ P(x) = \frac{1}{i!}\partial^i P(x,\dots,x) + Q(x)\] for some polynomial $Q$ of degree at most $i-1$ (this is the only point in the whole paper where we use the assumption that $|\F| > d \geq i$, in order to ensure invertibility of $i!$). It follows that 
\[ \rank_{i-1}(P) \ll_{k,\eps,\dim(\Sigma)} 1.\] This contradicts the $F$-regularity of the factor $\Factor$ if $F$ is assumed to grow sufficiently rapidly.\endproof

\section{Proof of Theorem \ref{wtn}}\label{proof-sec}

In this section we complete the proof of Theorem \ref{wtn}. Our starting point is the lemma of Bogdanov and Viola, stated as Lemma \ref{bog-viola-lem} in this paper. We urge the reader to recall the statement now. In view of that lemma, it suffices to establish the following proposition.

\begin{proposition}[Polynomials which are almost low-rank are low-rank]\label{99-to-100}
Suppose that $d \geq 1$ is an integer, and that Theorem \ref{wtn} holds for all orders up to $d-1$. Let $\sigma_d > 0$ be a small quantity to be specified later. Suppose that $P \in \mathcal{P}_d(\F^n)$ and that $\Factor$ is an $F$-regular factor of degree $d-1$. for some growth function which grows suitably rapidly in terms of $d$. Suppose that $\tilde P : \F^n \rightarrow \F$ is an $\Factor$-measurable function and that $\P(P(x) = \tilde P(x)) \geq 1 - \sigma_d$. Then $P$ is itself $\Factor$-measurable.
\end{proposition}
\emph{Proof of Theorem \ref{wtn} assuming Proposition \ref{99-to-100}.} This is almost immediate.  By induction we may fix $d \geq 1$ and assume that Theorem \ref{wtn} holds for all orders up to $d-1$. Take the function $\tilde P$ appearing in the conclusion of Lemma \ref{bog-viola-lem}. By construction, $\tilde P$ is measurable with respect to some factor  $\Factor_0$ of degree at most $d-1$ and dimension no more than $|\F|^5/\delta^2 \sigma$. By Lemma \ref{regular} we may extend $\Factor_0$ to a factor $\Factor$ which is $F$-regular and satisfies $\dim(\Factor) \ll_{F,d,\delta,\F} 1$. The function $\tilde P$ is manifestly $\Factor$-measurable, and so the result follows upon applying Proposition \ref{99-to-100}.\endproof

\emph{Proof of Proposition \ref{99-to-100}.} We use the same notation for the factor $\Factor$ that was introduced in Definition \ref{factor-def}. In particular this factor is defined by polynomials $P_{i,j} \in \mathcal{P}_i(\F^n)$: these should not be confused with the polynomial $P$ which is the subject of Proposition \ref{99-to-100}.

For the purposes of an initial discussion write $X$ for the set of points in $\F^n$ for which $P(x) = \tilde P(x)$, thus $|X| \geq (1 - \sigma_d)|\F^n|$. The key idea is that we may use $(d+1)$-dimensional parallelepipeds in $X$ to create new points $x'$ for which $P(x')$ does not depend on which atom of $\Factor$ the point $x'$ lies in. There are two procedures we might use:

\emph{1. Completing atoms.} Suppose that $x,h_1,\dots,h_{d+1}$ are such that all $2^{d+1}$ points $x + \omega \cdot h$ lie in the same atom $A$ of $\sigma(\mathcal{F})$. Suppose in addition that $x + \omega \cdot h \in X$ whenever $\omega \neq 0$. Then using the relation $\sum_{\omega} (-1)^{|\omega|} P(x + \omega \cdot h) = 0$ and the fact that $\tilde P$ is constant on $A$, we see that $x$ also lies in $X$.

\emph{2. Creating new atoms on which $P$ is constant.} Suppose that $A$ is an atom of $\sigma(\Factor)$ such that there are atoms $A_{\omega}$, $\omega \in \{0,1\}^{d+1} \setminus 0^{d+1}$ with the following property. For any $x \in A$, there are $h_1,\dots,h_{d+1} \in \F^n$ such that $x + \omega \cdot h \in A_{\omega}$ for all $\omega \in \{0,1\}^{d+1} \setminus 0$. Then if $P$ is constant on each of the $A_{\omega}$, it is also constant on $A$. This follows from the relation $\sum_{\omega} (-1)^{|\omega|} P(x + \omega \cdot h) = 0$ once again.

It is in fact possible to perform Procedures 1 and 2 simultaneously, but the exposition is fractionally clearer if the urge to do this is suppressed.

Let us start with an analysis of Procedure 1. It is easy to see using Lemma \ref{atom-size} that for $1 - O(\sqrt{\sigma_d})$ of the atoms in $\mathcal{B}$ we have $P_{x \in A}(P(x) = \tilde P(x)) \geq 1 - O(\sqrt{\sigma_d})$. We say that $P$ is \emph{almost constant} on such atoms, and our task is to show that $P$ is actually 100\% constant on each such atom.

Suppose that $P$ is almost constant on the atom $A = \Phi^{-1}(t)$, and write $A' \subseteq A$ for the set where $P = \tilde P$. 

\begin{lemma}[Avoiding bad parallelopipeds]\label{avoid-bad}  Let the notation and assumptions be as above.
Suppose that $\sigma_d$ is chosen sufficiently small. Fix an $x \in A$. Then there is $h$ so that all of the vertices $x + \omega \cdot h$, $\omega \neq 0^{d+1}$, lie in $A'$.
\end{lemma}

\proof
Let $N_\Box(x)$ denote the number of parallelopipeds $(x + \omega \cdot h)_{\omega \in \{0,1\}^{d+1}}$, all of whose vertices lie in $A$. 
The vector $(t,t,\dots,t) \in \Sigma^{\{0,1\}^{d+1}}$  trivially satisfies the parallelepiped constraints, and so by Proposition \ref{parallel-count} we have
\begin{equation}\label{nbox-size}
N_\Box(x) \sim |\F|^{n(d+1) - \sum_{i=1}^d M_i \sum_{1 \leq j \leq i} \binom{d+1}{j}}
\end{equation}
if $F$ is sufficiently rapidly growing. 

The number $N_\Box(x)$ of parallelopipeds in $A$ is thus quite large.  Unfortunately, this does not immediately imply that the number of paralleopipeds in $A'$ is large, as the $N_\Box(x)$ parallelopipeds in $A$ may all be intersecting the small set $A \backslash A'$.  However, it will turn out that such a concentration in $A \backslash A'$ can be picked up via the Cauchy-Schwarz inequality, as it will force into existence an anomalously large number of \emph{pairs} of parallelopipeds that share an additional vertex in common besides $x$.  The main difficulty in the proof then lies in counting number of such pairs properly.

We turn to the details.  It suffices to show, for each fixed $\omega_0 \in \{0,1\}^{d+1} \setminus 0^{d+1}$, that the number of parallelepipeds $(x + \omega \cdot h)_{\omega \in \{0,1\}^{d+1}}$, all of whose vertices lie in $A$, and with $x + \omega_0 \cdot h \in A \setminus A'$, is less than $2^{-d-2}N_{\Box}(x)$. The number of such ``bad'' parallelepipeds may be written as
\[ \sum_u 1_{A \setminus A'}(u)\sum_h 1_{x + \omega_0 \cdot h = u},\]
and we may use the Cauchy-Schwarz inequality to bound this above by
\[ |A \setminus A'|^{1/2} \big|\{(h,h') : x+ \omega \cdot h, x + \omega' \cdot h' \in A \;\; \mbox{for all $\omega, \omega' \in \{0,1\}^{d+1}$}, x + \omega_0 \cdot h = x + \omega_0 \cdot h' \}\big|^{1/2}.\]
Thus if $\sigma_d$ is chosen so small that $|A\setminus A'| \leq 2^{-2d-5}|A|$, it suffices to show that 
\begin{align}\nonumber \big|\{(h,h') : x+ \omega \cdot h, x + \omega' \cdot h' \in A \;\; \mbox{for all $\omega, \omega' \in \{0,1\}^{d+1}$}, & x + \omega_0 \cdot h = x + \omega_0 \cdot h' \}\big| \\ & \leq \frac{N_{\Box}(x)^2}{|A|}(1 + O(\eps))\label{eq330}\end{align}
for some sufficiently small $\eps > 0$.

By relabelling the cube $\{0,1\}^{d+1}$ if necessary, this may be recast as the problem of counting the number of $h,h' \in (\F^{n})^{d+1}$ satisfying the constraint 
\[ h_1 + \dots + h_s = h'_1 + \dots + h'_s\] and for which the two parallelepipeds
\[  \Box_1 := (x + \omega \cdot h)_{\omega \in \{0,1\}^{d+1}}\] and 
\[ \Box_2 := (x + \omega \cdot h')_{\omega \in \{0,1\}^{d+1}}\] lie in $A$. Substituting  \eqref{nbox-size} and the approximate size of $|A|$ (cf. Lemma \ref{atom-size}) into \eqref{eq330}, we see that our task is to establish that the number of such $h,h'$ is at most $1+ O(\eps)$ times $|\F|$ to the power $n(2d+1) + \sum_{i=1}^d M_i (1 - 2 \sum_{1 \leq j \leq i}\binom{d+1}{j})$. 

The parallelepipeds $\Box_1$ and $\Box_2$ share the common vertices $x$ and $x + h_1 + \dots + h_s$. Note that $\Box_1$ and $\Box_2$ may be embedded inside a $(2d+1)$-dimensional parallelepiped 
\[ \tilde \Box := (x + \omega \cdot y)_{\omega \in \{0,1\}^{2d+1}},\] where 
\[ y := (h_1,\dots,h_{s-1},h_s - h'_1 - \dots - h'_{s-1},h_{s+1},\dots, h_{d+1}, h'_1,\dots, h'_{s-1}, h'_{s+1},\dots, h'_{d+1}).\]
Thus, writing $\Box_1$ corresponds to the indices
\begin{equation}\label{eq111} \omega \in \{0,1\}^{d+1} \cdot (e_1,\dots,e_{s-1}, e_s + e_{d+2} + \dots + e_{d+s},e_{s+1},\dots,e_{d+1})\end{equation} and $\Box_2$ to the indices
\begin{equation}\label{eq112} \omega \in \{0,1\}^{d+1} \cdot (e_{d+2},\dots,e_{d+s},e_1 + \dots + e_s, e_{d+s+1},\dots, e_{2d+1})\end{equation}
where we use the usual dot product $(\omega_1,\ldots,\omega_{d+1}) \cdot (v_1,\ldots,v_{d+1}) := \omega_1 v_1 + \ldots + \omega_{d+1} v_{d+1}$.

Suppose that $i \in [d]$ and $j \in [M_i]$. Then $P_{i,j}(x + \omega \cdot y)$ is a polynomial of total degree at most $i$ in $\omega_1,\dots,\omega_{2d+1}$. Using the fact that $\omega = \omega^2 = \omega^3 = \dots$ for $\omega \in \{0,1\}$, we see that there exists a polynomial $Q_{i,j} : \Z^{2d+1} \rightarrow \F$ with total degree at most $i$ and degree at most $1$ in each of $\omega_1,\dots,\omega_{2d+1}$ with the property that 
\[ P_{i,j}(x + \omega \cdot y) = Q_{i,j}(\omega)\] for
$\omega \in \{0,1\}^{2d+1}$.  In fact this extension is unique, as the following lemma shows.

\begin{lemma}[Extension lemma]\label{ext-lem}
Suppose that $Q : \Z^k \rightarrow \F$ is a polynomial in variables $x_1,\dots,x_k$ of total degree with degree at most one in each $x_j$. Suppose that $Q(x_1,\dots,x_k)$ is equal to zero for $(x_1,\dots,x_k) \in \{0,1\}^k$. Then $Q \equiv 0$ identically.
\end{lemma}
\proof This appears, for example, as \cite[Lemma 2.1]{combinatorial-nullstellensatz}. We proceed by induction on $k$, the result being trivial when $k = 1$.  We may write
\[ Q(x_1,\dots,x_k) = R(x_1,\dots,x_{k-1}) + x_k S(x_1,\dots,x_{k-1})\] where both $R$ and $S$ have degree at most one in each $x_j$. Noting that $R(x_1,\dots,x_{k-1}) = Q(x_1,\dots,x_{k-1},0)$ and that $S(x_1,\dots,x_{k-1}) = Q(x_1,\dots,x_{k-1},1) - Q(x_1,\dots,x_{k-1},0)$, we see that $R(x_1,\dots,x_{k-1}) = S(x_1,\dots,x_{k-1}) = 0$ for all $x_j \in \{0,1\}$. By the inductive hypothesis this implies that $R \equiv S \equiv 0$ identically. \endproof

It follows from Lemma \ref{ext-lem}, \eqref{eq111}, \eqref{eq112} and the fact that $P_{i,j}(\Box_1)$ and $P_{i,j}(\Box_2)$ are fixed that $Q_{i,j}(\omega)$ is fixed for $\omega$ in both of the $d+1$-dimensional lattices
\[ \Lambda := \Z^{d+1} \cdot (e_1,\dots,e_{s-1},v,e_{s+1},\dots,e_{d+1})\]
and
\[ \Lambda' := \Z^{d+1} \cdot (e_{d+2},\dots,e_{d+s},v,e_{d+s+1},\dots,e_{2d+1}),\] where $v \in \Z^{2d+1}$ is the vector
\[ v := e_1 + \dots + e_s + e_{d+2} + \dots + e_{d+s}.\]
A second application of Lemma \ref{ext-lem}, noting that $2d > i$, confirms that $Q_{i,j}$ is determined on 
\[ \Z^{2d} \cdot (e_1,\dots,e_{s-1},e_{s+1},\dots,e_{2d+1}) + \{0,1\} \cdot v\]
by its values on
\[ S := \{0,1\} \cdot (e_1,\dots,e_{s-1},e_{s+1},\dots,e_{2d+1},v).\]
In particular we see that $Q_{i,j}(\omega)$, and hence $P_{i,j}(x + \omega \cdot y)$, is determined for $\omega \in \{0,1\}^{2d+1}$ by its values on $S$. Since $Q_{i,j}$ has degree at most $i$ we see that it is determined on $S$ by its values at arguments which are the sum of at most $i$ elements from $\{e_1,\dots,e_{s-1},e_{s+1},\dots,e_{2d+1},v\}$. 

Of the $\sum_{0 \leq j \leq i} \binom{2d+1}{j}$ possible choices for the values of the polynomials $Q_{i,j}$ at these arguments, $2\sum_{0 \leq j \leq i} \binom{d+1}{j} - 2$ of them are already fixed for us since $Q_{i,j}$ is fixed in both $\Lambda$ and $\Lambda'$. It follows that the number of choices of $(P_{i,j}(x + \omega \cdot y))_{\omega \in \{0,1\}^{2d+1}}$ is at most $|\F|$ to the power $1 + \sum_{1 \leq j \leq i} \binom{2d+1}{j} - 2\binom{d+1}{j}$. Summing over $i$ and $j$, it follows that the number of choices for $\Phi(\tilde \Box)$ subject to our constraints on $\Phi(\Box_1)$ and $\Phi(\Box_2)$ is at most $|\F|$ to the power $\sum_{i=1}^d M_i \big(1 + \sum_{1 \leq j \leq i} \binom{2d+1}{j} - 2\binom{d+1}{j}  \big)$.

For each such choice the number of $\tilde \Box$ is, by Proposition \ref{parallel-count}, $1 + O(\eps)$ times $|\F|$ to the power $n(2d+1) - \sum_{i=1}^{d-1} M_i \sum_{1 \leq j \leq i} \binom{2d+1}{j}$, and so the total number of $\tilde \Box$ is $1 + O(\eps)$ times $|\F|$ to the power $n(2d+1) + \sum_{i=1}^{d-1} M_i (1 - 2\sum_{1 \leq j\leq i} \binom{d+1}{j})$, which is what we wanted to prove. This concludes the proof of Lemma \ref{avoid-bad}.\endproof

Recall that $A' \subseteq A$ is the set of points where $P(x) = \tilde P(x)$. Now $A$ is an atom in the factor $\Factor$, which has degree $d-1$, and $P$ is a polynomial of degree $d$. We therefore see that if all the points $x + \omega \cdot h$, $\omega \in \{0,1\}^{d+1} \setminus 0^{d+1}$, lie in $A'$ then so does $x$. It follows from Lemma \ref{avoid-bad} that $A' = A$.

This completes the analysis of Procedure 1, and we find ourselves in the situation that $P(x) = \tilde P(x)$ on $1 - O(\sqrt{\sigma_d})$ of the atoms in $\sigma(\Factor)$. Call these the \emph{good} atoms. To perform procedure 2, we need only show that for any (bad) atom $A = A_0$ there are good atoms $A_{\omega}$, $\omega \in \{0,1\}^{d+1} \setminus 0^{d+1}$, such that the sequence of coordinates $t_{\Box} = \Phi(A_{\omega}) \in \Sigma^{\{0,1\}^{d+1}}$ satisfies the parallelepiped constraints. To do this it suffices to find just a single parallelepiped $(x + \omega \cdot h)_{\omega \in \{0,1\}^d}$ for which all of $x + \omega \cdot h$, $\omega \in \{0,1\}^{d+1} \setminus 0^{d+1}$, lie in good atoms. To see that this is possible, fix $x \in A_0$ and pick $h_1,\dots,h_{d+1}$ at random. It is clear that for any fixed $\omega \neq 0^{d+1}$, the probability that $x + \omega \cdot h$ lies in a good atom is the same as the probability that a random element of $\F^n$ lies in a good atom, which is $1 - O(\sqrt{\sigma_d})$ by Lemma \ref{atom-size}. If $\sigma_d \leq c2^{-2d}$ for sufficiently small $c$ it follows that there is indeed positive probability that all of the $x + \omega \cdot h$, $\omega \in \{0,1\}^{d+1} \setminus 0^{d+1}$, lie in good atoms.

We have now successfully performed Procedures 1 and 2. By earlier remarks, this concludes the proof of Proposition \ref{99-to-100} and hence, by the remarks at the start of the section, that of Theorem \ref{wtn}.\endproof

\section{Inverse theorems for the Gowers norm}

We can now give a fairly quick proof of Theorem \ref{main-2}.
We begin with a preliminary result which is already of interest.

\begin{proposition}\label{gowers-1} Suppose that $|\F| > d+1 \geq 2$ and that $\delta > 0$, let $P \in \mathcal{P}_{d+1}(\F^n)$, and write $f(x) := e_{\F}(P(x))$. Suppose that $\| f \|_{U^{d+1}} \geq \delta$.  Then $\rank_{d}(P) \ll_{d,\delta} 1$.  
\end{proposition}

\proof   
Write $\partial^{d+1} P(h_1,\dots,h_{d+1}) := D_{h_1} \dots D_{h_{d+1}}P(x)$. Since $P$ has degree $d+1$, this does not depend on $x$. From the definition of the $U^{d+1}$ norm, we have
$$ |\E_{h \in (\F^n)^{d+1}} e_\F(\partial^{d+1}P(h))| = \| f \|_{U^{d+1}}^{2^{d+1}} \geq \delta^{2^{d+1}}.$$
Applying Theorem \ref{wtn}, we conclude that
$$ \rank_{d}( \partial^{d+1} P ) \ll_{d,\delta} 1.$$
But since $|\F| > d + 1$ we have the Taylor expansion
\[ P(x) = \frac{1}{(d+1)!} \partial^{d+1} P(x,x,\dots,x) + Q(x),\]
where $\deg Q \leq d$. Thus the rank of $P$ is itself bounded by $O_{d,\delta}(1)$, as required.
\endproof

\emph{Proof of Theorem \ref{main-2}.} We fix $d$ and induct on $k$.  The cases $k \leq d$ are trivial (since $\|f\|_{u^{d+1}}=1$ in these cases), so we first verify the case $k=d+1$.  In this case, we know from Proposition \ref{gowers-1} that $\rank_{d}(P) \ll_{d,\delta} 1$, thus we can express $f(x) = e_{\F}(P(x))$ as some function of $O_{d,\delta}(1)$ polynomials of degree at most $d$.  By Fourier analysis, we can therefore obtain a representation
$$ f(x) = \sum_{j=1}^J c_j e_{\F}(Q_j(x))$$
where $J = O_{d,\delta}(1)$, $Q_j \in \mathcal{P}_{d}(\F^n)$, and $c_j$ are complex numbers of magnitude $O_{d,\delta}(1)$ for all $j \in [J]$.  It follows immediately that $f$ has inner product at $\gg_{d,\delta} 1$ with at least one of the functions $e_{\F}(Q_i(x))$, and therefore $\Vert f \Vert_{u^{d+1}} \gg_{d,\delta} 1$ as desired.

Now suppose that $k>d$ and the claim has already been proven for polynomials of degree $k$.  Suppose that $P \in \mathcal{P}_{k+1}(\F^n)$, that $f(x) := e_{\F}(P(x))$ and that $\Vert f \Vert_{U^{d+1}} \geq \delta$. By the monotonicity  of Gowers norms (see e.g. \cite[Chapter 11]{tao-vu}) we have
$$ \| f \|_{U^{k+1}} \geq \delta$$
and thus by Proposition \ref{gowers-1} we obtain
$$ \rank_{k}(P) \ll_{k,\delta} 1.$$
Let $F$ be a growth function (depending on $k,\delta,d$) to be chosen later.  Applying Lemma \ref{regular}, we can find an $F$-regular factor $\Factor$ of degree $k$ and dimension $O_{F,k,d,\delta}(1)$ such that $P$ is measurable with respect to $\sigma(\Factor)$.  By Fourier expansion, we can thus express
$$ f(x) = \sum_{Q_1 \in \Span(\Factor_1), \ldots, Q_{k} \in \Span(\Factor_{k})} c_{Q_1,\ldots,Q_{k}} e_\F( Q_1(x) + \ldots + Q_{k}(x) )$$
where the coefficients $c_{Q_1,\ldots,Q_{k}}$ are complex numbers of magnitude at most $B$ for some $B = O_{k,\dim(\Sigma)}(1)$.  We may use this expansion to split $f$ as $f_1 + f_2$, where
\begin{equation}\label{fup-def}
 f_1(x) := \sum_{Q_1 \in \Span(\Factor_1), \ldots, Q_d \in \Span(\Factor_d)} c_{Q_1,\ldots,Q_{d},0,\ldots,0} e_\F( Q_1(x) + \ldots + Q_{d}(x) )
\end{equation}
and
\begin{equation}\label{fu-def}
f_2(x) := \sum_{\substack{Q_1 \in \Span(\Factor_1), \ldots, Q_{k} \in \Span(\Factor_k) \\ Q_s \neq 0 \; \hbox{\scriptsize for some } s > d}} c_{Q_1,\ldots,Q_{k}} e_\F( Q_1(x) + \ldots + Q_{k}(x) ).
\end{equation}
Thus $f_2$ is the part of $f$ which ``genuinely has degree larger than $d$''. We shall show the $U^{d+1}$-norm of this part is small.

Suppose that polynomials $Q_1 \in \Span(\Factor_1),\ldots,Q_{k} \in \Span(\Factor_{k})$ are such that $Q_s$ is non-zero and $Q_{s+1},\ldots,Q_{k-1}$ all vanish for some $s > d$.  Since $\Factor$ is $F$-regular, we have $\rank_{s-1}(Q_{s}) \geq F(\dim(\Factor))$, and thus
\begin{equation}\label{souse}
 \rank_{s-1}( Q_1 + \ldots + Q_{k} - Q ) \geq F(\dim(\Factor)) - 1
\end{equation}
for any $Q \in \mathcal{P}_{d}(\F^n)$.
Applying Theorem \ref{wtn} and the induction hypothesis, we conclude (if $F$ is large enough) that
$$ \| e_\F( Q_1 + \ldots + Q_{k} ) \|_{U^{k+1}} \leq \frac{\delta}{2B |\Factor_1| \ldots |\Factor_{k}|}.$$
Since the Gowers $U^{k+1}$-norm obeys the triangle inequality (see e.g. \cite[Lemma 3.9]{gowers}), it follows that $\| f_2\|_{U^{k+1}} \leq \delta/2$. Recalling that $\Vert f \Vert_{U^{k+1}} \geq \delta$, another application of the triangle inequality implies that $\Vert f_1 \Vert_{U^{k+1}} \geq \delta/2$.
Now by Cauchy-Schwarz we have
$$ \| f_1 \|_{U^{k+1}}^{2^{k+1}} \leq \| f_{1} \|_{2}^2 \| f_{1} \|_{\infty}^{2^{k+1}-2}.$$
From the bounds on the Fourier coefficients $c_{Q_1,\dots,Q_k}$ we have $\| f_1 \|_{\infty} \ll_{k,\dim(\Factor)} 1$, and therefore
$$ \langle f_1, f_1 \rangle = \| f_1 \|_{2}^2 \gg_{d,k,\delta,\dim(\Factor)} 1.$$
From \eqref{fup-def} and the pigeonhole principle it follows that there exist $Q_1 \in \Factor_1, \ldots, Q_{d} \in \Factor_{d}$ such that
$$ |\langle f_1, e_\F( Q_1 + \ldots + Q_{d} ) \rangle| \geq \eps$$
for some $\eps \gg_{d,k,\delta,\dim(\Sigma)} 1$. On the other hand, from \eqref{souse}, Theorem \ref{wtn}, and \eqref{fu-def} we have
$$ |\langle f_2, e_\F( Q_1 + \ldots + Q_{d} ) \rangle| \leq \eps/2$$
if $F$ grows sufficiently rapidly. Hence from one further application of the triangle inequality we have
\[ |\langle f, e_{\F}(Q_1 + \dots + Q_d)\rangle| \geq \eps/2,\] and thus $\Vert f \Vert_{u^d} \geq \eps/2$. Therefore the induction goes through and we have proved Theorem \ref{main-2}.
\endproof

\section{A recurrence result}

Proposition \ref{99-to-100} had a rather lengthy proof.  However, the claim is much simpler in the case when the factor $\Factor$ is trivial.  More precisely, we have the following slight generalization of \cite[Proposition 4.5]{tao-stoc}. 

\begin{lemma}[Non-zero polynomials do not vanish almost everywhere]\label{nonzero}
Suppose that $P \in \mathcal{P}_d(\F^n)$ and that $\P_{x \in \F^n}(P(x) = 0) > 1 - 2^{-d}$. Then $P$ is identically zero.
\end{lemma}

\begin{remark} This lemma is almost certainly folkloric, but we do not have a precise reference for it.
\end{remark}

\proof We proceed by induction on $d$, the result being obvious for $d = 1$. For any fixed $h$ we have $\P_{x \in \F^n}(P(x+h) = P(x) = 0) > 1 - 2^{-(d-1)}$. Applying the inductive hypothesis to $P(x+h) - P(x) \in \mathcal{P}_{d-1}(\F^n)$, we see that $P(x+h) - P(x) = 0$ for all $x,h$. This manifestly implies the result.\endproof

A short consequence of Lemma \ref{nonzero} is the following curious recurrence result. 

\begin{lemma}[Multiple polynomial recurrence]\label{mpr}  Suppose that $d, k \geq 1$ are integers, that $P_1,\ldots,P_k \in \mathcal{P}_d(\F^n)$ are polynomials and that $x_0 \in \F^n$.  Then
$$ \P_{x \in \F^n} ( P_i(x) = P_i(x_0) \hbox{ for all } i = 1,\dots,k ) \geq 2^{-(|\F|-1)kd}.$$
\end{lemma}

\begin{proof} Consider the polynomial
\[ Q(x) := \prod_{i=1}^k \prod_{\substack{t \in \F\\ t \neq P_i(x_0)}} (P_i(x) - t).\]
This polynomial has degree $(|\F| - 1)kd$, and clearly $Q(x_0) \neq 0$.  Applying Lemma \ref{nonzero} in the contrapositive, we conclude
$$ \P_x (Q(x) \neq 0) \geq 2^{-(|\F|-1)kd}$$
and the claim follows.
\end{proof}

\begin{remark} In the case $d < |\F|$, one could also obtain a qualitative version of Lemma \ref{mpr} by combining Lemma \ref{regular} (applied to the factor generated by $P_1,\ldots,P_k$) followed by Lemma \ref{atom-size}.  Of course, the bounds obtained by this approach are far weaker.
\end{remark}

\section{Representations that respect degree}

The results of this section and the next are somewhat technical, and by necessity some of the notation is a little fearsome. First-time readers may wish to skip to the discussion of the counterexample of Theorem \ref{main-1}, which is presented in \S \ref{counter-sec}. 

In previous sections we showed discussed the notion of low-rank polynomials $P \in \mathcal{P}_d(\F^n)$, which can be expressed as $B(Q_1,\dots,Q_k)$ with $Q_i \in \mathcal{P}_{d-1}(\F^n)$. In this section we show how (under a regularity assumption on the factor generated by the $Q_i$) the function $B$ can be chosen to be a polynomial with controlled degree.

\begin{definition}  Let $\Factor$ be a factor of degree $d \geq 1$ on a $\F^n$.  A \emph{$\Factor$-monomial} is any product of the form $\prod_{j=1}^J Q_j$, where each $Q_j$ belongs to one of the vector spaces $\Span(\Factor_{d_j})$ for some $d_j \in \{1,\ldots,d\}$.  The \emph{$\Factor$-degree} of the $\Factor$-monomial $\prod_{j=1}^J Q_j$ is defined to be $\sum_{j=1}^J d_j$.  If $D \geq 0$, we define a \emph{$\Factor$-polynomial} of $\Factor$-degree at most $D$ to be any linear combination of $\Factor$-monomials of $\Factor$-degree at most $D$.
\end{definition}

\begin{example} Let $\F$ have large characteristic.  If $\Factor$ is the degree 2 factor on $\F^5$ consisting of the four polynomials $X_1X_2 + X_3$, $X_1X_2 + X_4$, $X_2 + X_3$ and $X_1 + X_5$, where $X_1,\ldots,X_5$ are the coordinate functions, the polynomial $(X_1X_2 + X_3)(X_1 + X_5)^7 + (X_1 + X_2 + X_3 + X_5)^9$ has $\Factor$-degree 9, and so does $(X_3 - X_4)^4(X_2 + X_3)$, since $X_3 - X_4 \in \Span(\Factor_2)$.  
\end{example}

In the above example we saw that the $\Factor$-degree of a polynomial can exceed the ordinary degree due to dependencies among the polynomials in the factor.  The following theorem can be viewed as a converse to this phenomenon.

\begin{theorem}[Degree and $\Factor$-degree agree for regular factors]\label{repr}  Let $0 \leq d, D < |\F|$.  Then there exists a growth function $F$ \textup{(}depending on $d$ and $D$\textup{)} with the following property. Suppose that $P \in \mathcal{P}_D(\F^n)$ is measurable with respect to $\sigma(\Factor)$, where $\Factor$ is an $F$-regular factor of degree $d$ on $\F^n$. Then $P$ has $\Factor$-degree at most $D$.
\end{theorem}

\proof   Let $d,D$ be as above, let $F$ be a rapid growth function to be chosen later, and let $P, \Factor$ be as above.  Since $P$ is measurable with respect to $\sigma(\Factor)$, we have a representation
$$ P = B( P_{1,1},\ldots,P_{1,M_1},\ldots,P_{d,1},\ldots,P_{d,M_d} )$$
for some function $B: \Sigma \to \F$.  As $\F$ is a finite field, we can view $B$ as a polynomial of $\dim(\Factor)$ variables, which has individual degree at most $|\F|-1$ in each of the variables (note that all higher degrees can be eliminated since $x^{|\F|} = x$).  Thus we can write
\begin{equation}\label{pex}
 P = \sum_{ r \in R} c_r \prod_{i=1}^d \prod_{j=1}^{M_i} P_{i,j}^{r_{i,j}}
\end{equation}
where $R$ is the set of all tuples $r = (r_{i,j})_{1 \leq i \leq d; 1 \leq j \leq M_i}$, and the $c_r$ are coefficients in $\F$.  

For each tuple $r \in R$, we define the \emph{weight} $|r|$ of $r$ by the formula
$$ |r| := \sum_{i=1}^d i \sum_{j=1}^{M_i} r_{i,j}.$$
To prove the claim, it suffices to show that $c_r=0$ for all tuples $r$ with weight larger than $D$.  Suppose for contradiction that this is not the case.  Then we can find $r$ with $|r| > D$ such that $c_r \neq 0$; without loss of generality we may assume that $|r|$ is maximal with respect to this property.  From \eqref{pex}, we thus have
$$ P(x) = c_r \prod_{i=1}^d \prod_{j=1}^{M_i} P_{i,j}(x)^{r_{i,j}} 
+ \sum_{s \in R \backslash \{r\}: |s| \leq |r|}
c_s \prod_{i=1}^d \prod_{j=1}^{M_i} P_{i,j}(x)^{s_{i,j}}$$
for all $x \in \F^n$.  Since $P$ has degree $D < |r|$, its $|r|^{\th}$ order derivatives vanish.  Thus we have
\begin{align*}
0 &= c_r \sum_{ \omega \in \{0,1\}^{|r|} } (-1)^{|\omega|} \prod_{i=1}^d \prod_{j=1}^{M_i} P_{i,j}(x+\omega \cdot h)^{r_{i,j}} \\
&\quad + 
\sum_{s \in R \backslash \{r\}: |s| \leq |r|}
\sum_{ \omega \in \{0,1\}^{|r|} } (-1)^{|\omega|} 
c_s \prod_{i=1}^d \prod_{j=1}^{M_i} P_{i,j}(x+\omega \cdot h)^{s_{i,j}}
\end{align*}
for all $x \in \F^n$ and $h \in (\F^n)^{|r|}$.  

Now if $a = (a_{i,j}(\omega)) \in \Sigma^{\{0,1\}^{|r|}}$ satisfies the parallelelepiped constraints, and if $F$ grows sufficiently rapidly, then we know from Proposition \ref{parallel-count} that there are $x \in \F^n$ and $h \in (\F^n)^{|r|}$ such that $P_{i,j}(x+\omega \cdot h) = a_{i,j}(\omega)$ for all $i,j$ with $i \in [d]$ and $j \leq M_i$ and for all $\omega \in \{0,1\}^{|r|}$.  We thus conclude that
\begin{align*}
0 &= c_r \sum_{ \omega \in \{0,1\}^{|r|} } (-1)^{|\omega|} \prod_{i=1}^d \prod_{j=1}^{M_i} a_{i,j}(\omega)^{r_{i,j}} \\
&\quad + 
\sum_{s \in R \backslash \{r\}: |s| \leq |r|}
\sum_{ \omega \in \{0,1\}^{|r|} } (-1)^{|\omega|} 
c_s \prod_{i=1}^d \prod_{j=1}^{M_i} a_{i,j}(\omega)^{s_{i,j}}
\end{align*} for all $a \in \Sigma^{\{0,1\}^{|r|}}$ satisfying the parallelepiped constraints.  Thus, to obtain the desired contradiction, it will suffice to locate such an $a$ for which 
\begin{equation}\label{qbig}
\sum_{ \omega \in \{0,1\}^{|r|} } (-1)^{|\omega|} \prod_{i=1}^d \prod_{j=1}^{M_i} a_{i,j}(\omega)^{r_{i,j}} \neq 0
\end{equation}
but such that
\begin{equation}\label{qsmall}
\sum_{ \omega \in \{0,1\}^{|r|} } (-1)^{|\omega|} \prod_{i=1}^d \prod_{j=1}^{M_i} a_{i,j}(\omega)^{s_{i,j}} = 0
\end{equation}
for all $s \in R \backslash \{r\}$ with $|s| \leq |r|$.

We can do this explicitly as follows.  Let us parametrise $\{0,1\}^{|r|}$ as $\prod_{i=1}^d \prod_{j=1}^{M_i} (\{0,1\}^i)^{r_{i,j}}$, thus we write each $\omega \in \{0,1\}^{|r|}$ as 
$\omega_{i,j,k,t}$, where $1 \leq i \leq d$, $1 \leq j \leq M_i$, $1 \leq k \leq i$ and $1 \leq t \leq r_{i,j}$. Define $a \in \Sigma^{\{0,1\}^{|r|}}$ by
\[ a_{i,j}(\omega) := \sum_{t=1}^{r_{i,j}} \prod_{k=1}^i \omega_{i,j,k,t},\] where we embed $\{0,1\}$ into $\F$ in the obvious way. Since $a_{i,j}(\omega)$ is a linear combination of products of $i$ coordinates of $\omega$, it is easy to see that $a$ satisfies the parallelepiped constraints.

Let us now verify \eqref{qsmall}. For fixed $i,j$,  $a_{i,j}(\omega)$ depends only on the components lying in $(\{0,1\}^i)^{r_{i,j}}$, which are disjoint as $i,j$ vary. We can therefore factorise the left-hand side of \eqref{qsmall} (with a hopefully obvious notation) as
$$
\prod_{i=1}^d \prod_{j=1}^{M_i} \big(\sum_{\eta \in (\{0,1\}^i)^{r_{i,j}}} (-1)^{|\eta|}a_{i,j}(0,\dots,0,\eta,0,\dots,0)^{s_{i,j}}   \big),
$$
where the notation is supposed to suggest that $\eta$ is in the $i,j$-part of the product $\prod_{i=1}^{d} \prod_{j =1}^{M_i} (\{0,1\}^i)^{r_{i,j}}$.
On the other hand,
If $|s| \leq |r|$ and $s \neq r$, then from the pigeonhole principle there must be some $i \leq d$ and some $j \leq M_i$ such that $s_{i,j} < r_{i,j}$.  Fixing this $i,j$, it thus suffices to show that
\[ \sum_{\eta \in (\{0,1\}^i)^{r_{i,j}}} (-1)^{|\eta|}a_{i,j}(0,\dots,0,\eta,0,\dots,0)^{s_{i,j}} = 0.\]
But we observe that $a_{i,j}(\omega)^{s_{i,j}}$ is a linear combination of products of $i s_{i,j}$ coordinates of $\omega$, which is strictly less than $i r_{i,j}$, and the claim follows.

Now we verify \eqref{qbig}.  Performing the same factorisation as before, it suffices to show that
\begin{equation}\label{arj}
\sum_{\eta \in (\{0,1\}^i)^{r_{i,j}}} (-1)^{|\eta|}a_{i,j}(0,\dots,0,\eta,0,\dots,0)^{r_{i,j}} \neq 0
\end{equation}
for each $i,j$.  But $a_{i,j}(0,\dots,0,\eta,0,\dots,0)^{r_{i,j}}$ is equal to $r_{i,j}! \prod_{k=1}^i \prod_{t=1}^{r_{i,j}} \eta_{k,t}$ (viewed of course as an element of $\F$), plus several other monomials, none of which involve all of the $\eta_{k,t}$. From this we see that the left-hand side
of \eqref{arj} is simply $(-1)^{i r_{i,j}} r_{i,j}!$.  Since $r_{i,j} < |\F|$, this expression is non-zero in $\F$, as desired.
\endproof

Combining this theorem with Lemma \ref{regular} we immediately obtain the following corollary.

\begin{corollary}[Minimal-degree representation of polynomials]\label{repr-cor}  Let $1 \leq d, D < |\F|$, and let $F$ be a growth function.  Then whenever $P \in \mathcal{P}_D(\F^n)$ is measurable with respect to a factor $\Factor$ of order $d$ on $\F^n$, there exists an $F$-regular extension $\Factor'$ of $\Factor$ of order $d$ with $\dim(\Factor') \ll_{d,D,\dim(\Factor)} 1$ such that $P$ has $\Factor'$-degree at most $D$.
\end{corollary}

\section{A nullstellensatz}

In this section we establish a kind of finite field analogue of Hilbert's Nullstellensatz. These results are not needed elsewhere in the paper, but are illustrative applications of the previous machinery, and may be of some independent interest.

\begin{proposition}[Nullstellensatz]  Let $k \geq 0$ and $0 \leq d < |\F|$, and let $P_1,\ldots,P_k \in \mathcal{P}_d(\F^n)$.  Let $Q \in \mathcal{P}_d(\F^n)$ be such that $Q$ vanishes whenever $P_1,\ldots,P_k$ all vanish.
Then there exist polynomials $R_1,\ldots,R_k$ of degree $O_{d,k}(1)$ such that
$$ Q(x) = P_1(x) R_1(x) + \ldots + P_k(x) R_k(x)$$
for all $x \in \F^n$.
\end{proposition}

\proof Let $\Factor$ be the degree $d$ factor defined by the polynomials $P_1,\ldots,P_k,Q$.  Let $F$ be a growth function to be chosen later.  By Lemma \ref{regular}, we can extend $\Factor$ to an $F$-regular factor $\Factor'$ of order $d$ and dimension $O_{d,k,F}(1)$.  If $F$ is sufficiently rapid, then by Lemma \ref{atom-size} we see that the configuration map $\Phi' : \F^n \rightarrow \Sigma'$ corresponding to $\Factor'$ is surjective.  Since $P_1,\ldots,P_k,Q$ are measurable with respect to $\sigma(\Factor')$, we can write $P_i = p_i \circ \Phi'$ and $Q = q \circ \Phi'$ for some $p_i, q: \Sigma' \to \F$. Our assumption together with the surjectivity of $\Phi'$ implies that if $z \in \Sigma'$ is such that $p_i(z) = 0$ for $i=1,\dots,k$ then $q(z) = 0$.
By working on each point $z$ separately, one can therefore find functions $r_1,\ldots,r_k: \Sigma' \to \F$ such that
$$ q(z) = p_1(z) r_1(z) + \ldots + p_k(z) r_k(z)$$
for all $z \in \Sigma'$.
Composing with $\Phi'$ we conclude that
$$ Q(x) = P_1(x) R_1(x) + \ldots + P_k(x) R_k(x)$$
for all $x \in \F^n$, where $R_i := r_i \circ \Phi'$.
As $\Sigma'$ has dimension $O_{d,k,F}(1)$, one can view $r_1,\ldots,r_k$ as polynomials of degree $O_{d,k,F}(1)$, and so $R_1,\ldots,R_k$ are also polynomials of degree $O_{d,k,F}(1)$.  The claim follows.
\endproof

In the above result the polynomials $R_i$ had bounded degree. However, if the polynomials $P_1,\dots,P_k$ arose from a sufficiently regular factor, one can get the sharp degree bound for $R_i$, namely $\deg(R_i) = \deg(Q) - \deg(P_i)$.

\begin{proposition}[Exact nullstellensatz]\label{exact-nullstellensatz}  Let $D, d, k \geq 0$.  Then there exists a growth function $F$ \textup{(}depending on $D,d,k$\textup{)} with the following property: given any $F$-regular factor $\Factor$ of order $d$ and dimension at most $D$ on $\F^n$, and given any $Q \in \mathcal{P}_k(\F^n)$ which vanishes whenever the polynomials $P_{i,j}$ defining $\Factor$ all vanish,
there exist polynomials $R_{i,j} \in \mathcal{P}_{k-i}(\F^n)$ for all $i \leq \min(d,k)$ and $j \leq M_i$ such that
$$ Q(x) = \sum_{i=1}^{\min(d,k)} R_{i,j}(x)P_{i,j}(x)$$
for all $x \in \F^n$.
\end{proposition}
Before embarking on the proof, we give a technical generalisation of the regularity lemma, Lemma \ref{regular}. Let us say that an extension $\Factor'$ of a factor $\Factor$ of order $d$ is \emph{non-disruptive} if we have $\Factor_i \subseteq \Factor'_i$ for all $i = 1,\dots,d$.  Clearly if $\Factor'$ is a non-disruptive extension of $\Factor$ and $\Factor'$ is $F$-regular, then $\Factor$ must also be $F$-regular.  Our next lemma can be regarded as a kind of converse to this fact.

\begin{lemma}[Relative regularity lemma]\label{regular-1} Let $d, D \geq 1$ and let $F$ be a growth function.  Then there exists a growth function $\tilde F$ such that whenever $\Factor$ is a $\tilde F$-regular factor of order $d$ on $\F^n$, and $\Factor'$ is an extension of $\Factor$ of dimension at most $D$, there exists an $F$-regular extension $\Factor''$ of $\Factor'$ with the dimension bound
\begin{equation}\label{dimsoc}
 \dim(\Factor'') \ll_{F,d,D} 1
\end{equation}
such that $\Factor''$ is a non-disruptive extension of $\Factor$.
\end{lemma}

\proof   Fix $d,F$, and let $\tilde F$ be a sufficiently rapid growth function to be chosen later. First observe that as the polynomials in $\Factor$ are $\Factor'$-measurable, we have the crude bound $\dim(\Factor) \ll_{D}1$, and so we may allow our constants to depend on $\dim(\Factor)$ also.

By replacing $\Factor'_i$ with $\Factor'_i \cup \Factor_i$ for $1 \leq i \leq d$ if necessary (and increasing $D$ accordingly) we may assume that $\Factor'$ is a non-disruptive extension of $\Factor$.  We now keep $\Factor$ fixed and induct on the dimension vector $(\dim(\Factor'_1),\ldots,\dim(\Factor'_d))$ of $\Factor'$ in exactly the same way as in Lemma \ref{regular} in order to obtain an $F$-regular extension $\Factor''$ of $\Factor'$ obeying
\eqref{dimsoc}.  The key point is that the low-rank polynomials $Q_i$ which arise in the proof of Lemma \ref{regular} can never arise from $\Factor_i$ if $\tilde F$ is chosen sufficiently rapid (thanks to \eqref{dimsoc}).   Because of this, we can easily arrange that the extension $\Factor''$ appearing in the proof of Lemma \ref{regular} continues to be a non-disruptive extension of $\Factor$, and the claim easily follows.
\endproof

\emph{Proof of Proposition \ref{exact-nullstellensatz}.}
Fix $D,d,k \geq 0$. By adding dummy polynomials to $\Factor$ and enlarging $d$ if necessary we may assume that $d \geq k$. Let $F_1$ be a growth function depending on $D,d,k$ to be chosen later, and let $F$ be an even more rapid growth function depending on $D,d,k,F_1$ and also to be chosen later.

Let $\Factor, Q$ be as in the statement of the proposition.  Let $\Factor' = (P_{i,j})_{i \in [d], j \leq M'_i}$ be the factor of order $d$ formed by adjoining $Q$ to $\Factor$.  Applying Lemma \ref{regular-1}, we see (if $F$ is sufficiently rapid depending on $D,d,k,F_1$) that we can find an $F_1$-regular extension $\Factor'' = (P_{i,j})_{i \in [d], j \leq M''_i}$ of $\Factor$ of order $\max(d,k)$ which is a non-disruptive extension of $\Factor$.  Applying Theorem \ref{repr}, we conclude (if $F_1$ is sufficiently rapid depending on $D,d,k$) that $Q$ has $\Factor''$-degree at most $k$. Using the identity $x^{|\F|} = x$ to eliminate all exponents greater than or equal to $|\F|$, we have a representation $Q(x) = q(\Phi''(x))$ for all $x \in \F^n$, where $q: \Sigma'' \to \F$ is a polynomial which takes the form
\begin{equation}\label{qex}
 q(t) := \sum_{ s \in S_k } c_s \prod_{i=1}^d \prod_{j=1}^{M''_i} t_{i,j}^{s_{i,j}} 
 \end{equation}
where $c_s \in \F$ for all $s \in S_k$, and $S_k$ is the collection of all tuples $( s_{i,j} )_{1 \leq i \leq d; 1 \leq j \leq M''_i}$ of non-negative integers $0 \leq s_{i,j} < |\F|$ obeying the weight condition  
$$ \sum_{i=1}^d \sum_{j=1}^{M''_i} i s_{i,j} \leq k.$$

By hypothesis, $Q(x)$ vanishes whenever all the $P_{i,j}(x)$ vanish for $i = 1,\dots,d$ and $j \leq M_i$.  On the other hand, by Lemma \ref{atom-size} we see (if $F_1$ is sufficiently rapid) that $\Phi'' : \F^n \rightarrow \Sigma''$ is surjective.  We conclude that $q$ vanishes on the coordinate subspace
$$ W := \{ t \in \Sigma'': t_{i,j} = 0 \hbox{ for all } i = 1,\dots,d \hbox{ and } j \leq M'_i \}.$$
Restricting $q$ to $W$ and then equating coefficients (recalling from the Lagrange interpolation formula that the coefficients are uniquely determined as long as all exponents are less than $|\F|$) we conclude that $c_s$ vanishes for each $s \in S$ such that $s_{i,j} = 0$ for all $i,j$ with $i \leq d$ and $j \leq M_i$.  From this, we can easily obtain a representation of the form
$$ q(t) = \sum_{i=1}^d \sum_{j=1}^{M_i} t_{i,j} r_{i,j}(t)$$
where each $r_{i,j}$ has weighted degree at most $k-i$ in the sense that it can be expanded into monomials as in \eqref{qex} but using only exponents from $S_{k-i}$ rather than all of $S_k$.  In particular $r_{i,j}$ must vanish for $i > k$.  Substituting $t = \Phi''(x)$ we obtain the claim.\endproof

\section{The counterexample}\label{counter-sec}

In this section we analyse the counterexample to the inverse conjecture for the Gowers norms in characteristic two by proving Theorem \ref{main-1}. Recall what is claimed in that theorem: the elementary symmetric quartic
\[ S_4(x) = \sum_{1 \leq i_1 < i_2 < i_3 < i_4 \leq n} x_{i_1}x_{i_2}x_{i_3}x_{i_4}\] is such that $f(x) = (-1)^{S_4(x)}$ has large $U^4$-norm on $\F_2^n$, but this function does not correlate well with any cubic phase.

We begin by establishing that the $U^4$-norm of this function is large.
Define the symmetric bilinear form $B: \F_2^n \times \F_2^n \to \F_2$ by
\begin{equation}\label{bab}
B( a, b ) := \sum_{1 \leq i,j \leq n: i \neq j} a_i b_j 
\end{equation}
for $a = (a_1,\ldots,a_n), b = (b_1,\ldots,b_n)$ in $\F_2^n$. One readily verifies\footnote{For a generalisation of this identity, see Lemma \ref{symderiv} below.} the identity
\begin{align}\nonumber
D_a D_b D_c D_d S_4(x) & = \sum_{1 \leq i,j,k,l \leq n: i,j,k,l\; \mbox{\scriptsize distinct }} a_i b_j c_k d_l \\ & = B(a,b) B(c,d) + B(a,c) B(b,d) + B(a,d) B(b,c),\label{qident}
\end{align}
and so
\begin{equation}\label{u4a} \Vert f \Vert_{U^4}^4 = 
\E_{a,b,c,d \in \F_2^n} (-1)^{B(a,b) B(c,d) + B(a,c) B(b,d) + B(a,d) B(b,c)}.
\end{equation}
To compute this quantity, we will need to look at the distribution of the sextuplet 
\begin{equation}\label{q6}
B_6(a,b,c,d) := (B(a,b), B(a,c), B(a,d), B(b,c), B(b,d), B(c,d)) \in \F_2^6
\end{equation}
as $a,b,c,d$ vary in $\F_2^n$.  This distribution can be controlled by standard Gauss sum estimates such as the following (cf. also Lemma \ref{gauss}).

\begin{lemma}[Gauss sum estimate]  For any $\xi_{ab}, \xi_{ac}, \xi_{ad}, \xi_{bc}, \xi_{bd}, \xi_{cd} \in \F_2$, not all zero, we have
$$ \E_{a,b,c,d \in \F_2^n} (-1)^{\xi_{ab} B(a,b) + \xi_{ac} B(a,c) + \xi_{ad} B(a,d) + \xi_{bc} B(b,c) + \xi_{bd} B(b,d) + \xi_{cd} B(c,d)} = O(2^{-n/2}).$$
\end{lemma}

\proof   By symmetry we may assume $\xi_{ab} = 1$.  It suffices to show that
$$ \E_{a,b \in \F_2^n} (-1)^{B(a,b) + \xi_{ac} B(a,c) + \xi_{ad} B(a,d) + \xi_{bc} B(b,c) + \xi_{bd} B(b,d) + \xi_{cd} B(c,d)} = O( 2^{-n/2} )$$
uniformly in $c,d \in \F_2^n$.  But if we fix $c,d$, we can write the left-hand side as
$$ \E_{a,b \in \F_2^n} (-1)^{B(a,b) + L(a) + L'(b)}$$
for some $L, L' \in \mathcal{P}_1(\F_2^n)$.  Applying Cauchy-Schwarz to eliminate the $(-1)^{L'(b)}$ factor, we can estimate this quantity in absolute value by
$$ |\E_{a,a',b \in \F_2^n} (-1)^{B(a,b) - B(a',b) + L(a)-L(a')}|^{1/2};$$
writing $c := a-a'$ this becomes
$$ |\E_{c,b \in \F_2^n} (-1)^{B(c,b) + L(c)}|^{1/2}.$$
Performing the $c$ average using Fourier analysis and using the triangle inequality, we can bound this by
$$ |\P_{c \in \F_2^n}( B(c,b) = 0 \hbox{ for all } b \in \F_2^n )|^{1/2}.$$
But $B$ has rank $n-O(1)$, and so 
$$ \P_{c \in \F_2^n}( B(c,b) = 0 \hbox{ for all } b \in \F_2^n ) = O( 2^{-n} ).$$
The claim follows.
\endproof

From this lemma and Fourier analysis on $\F_2^6$ (as in the proof of Lemma \ref{atom-size}) we see that $B_6$ is equidistributed in the sense that
$$ \P_{a,b,c,d \in \F_2^n}(B_6(a,b,c,d) = q) = 2^{-6} + O(2^{-n}) \hbox{ for all } q \in \F_2^6.$$
It follows that \eqref{u4a} can be rewritten as
$$ \E_{q_{ab}, q_{ac}, q_{ad}, q_{bc}, q_{bd}, q_{cd} \in \F_2} (-1)^{q_{ab} q_{cd} + q_{ac} q_{bd} + q_{ad} q_{bc}} + O(2^{-n}).$$
But we can factorise the expectation and rewrite this expression as
$$ (\E_{q,q' \in \F_2} (-1)^{qq'})^3 + O(2^{-n/2}).$$
Since $\E_{q,q' \in \F_2} (-1)^{qq'} = \frac{1}{2}$, it follows that $\Vert f \Vert_{U^4}^4 = \frac{1}{8} + O(2^{-n})$ as asserted in \eqref{u4} of Theorem \ref{main-1}.

Now we turn to \eqref{u4b}, which asserts that $f$ does not have substantial correlation with a cubic phase. Let us remind the reader once more that a better bound is contained in the independent work of Lovett, Meshulam and Samorodnitsky \cite{lms}. Our bound is all but contained in Alon and Beigel \cite[Theorem 7]{beigel}, although we recall that argument here for the convenience of the reader.

If $x = (x_1,\ldots,x_n) \in \F_2^n$, let $|x|$ denote the number of indices $i \in [n]$ for which $x_i = 1$. It is clear that $S_d(x) = \binom{|x|}{d} \md{2}$. Recalling Lucas' theorem on binomial coefficients $\md{p}$, which states that
\begin{equation}\label{lucas}
\binom{a}{b} \equiv \binom{a_0}{b_0} \dots \binom{a_k}{b_k} \md{p}
\end{equation}
whenever $a = a_0 + a_1 p + a_2 p^2 + \dots + a_kp^k$ and $b = b_0 + b_1p + b_2p^2 + \dots + b_kp^k$ with $0 \leq a_i, b_i < p$, 
we see that
\begin{align*}
S_0(x) &= 0 \\
S_1(x) &= 1 \hbox{ iff } |x| \equiv 1 \md{2} \\
S_2(x) &= 1 \hbox{ iff } |x| \equiv 2,3 \md{4} \\
S_3(x) &= 1 \hbox{ iff } |x| \equiv 3 \md{4} \hbox{ and } \\
S_4(x) &= 1 \hbox{ iff } |x| \equiv 4,5,6,7 \md{8}.
\end{align*}
On the other hand we have, by a technique once known\footnote{One can also interpret this computation as exhibiting (by the usual Fourier-analytic method) the exponential mixing rate of a simple random walk on $\Z/8\Z$.} as ``multisection of series'',
\begin{align*}
\P_{x \in \F_2^n}(|x| \equiv a \md{8}) &= 2^{-n}\!\!\!\!\sum_{j \equiv a \mdsub{8}} \binom{n}{j} \\
&= \frac{1}{8} \sum_{r = 0}^7 e^{-2\pi i ra/8}\big( \frac{1 + e^{2\pi i r/8}}{2}\big)^n \\
&= \frac{1}{8} + O(2^{-\Omega(n)}).
\end{align*}
From these facts and some computation we easily conclude that
$$ \E_{x \in \F_2^n} (-1)^{S_4(x) + c_3 S_3(x) + c_2 S_2(x) + c_1 S_1(x) + c_0 S_0}  = O(2^{-\Omega(n)})$$
for all coefficients $c_0,c_1,c_2,c_3 \in \F_2$.  Clearly this immediately implies that
\begin{equation}\label{base}
\E_{x \in\F_2^n}(-1)^{S_4(x) + c_3 S_3(x) + c_2 S_2(x) + c_1 S_1(x) + c_0 S_0 - Q_0}  = O(2^{-\Omega(n)})
\end{equation}
whenever $Q_0 \in \mathcal{P}_0(\F_2^n)$ and $c_0,c_1,c_2,c_3 \in \F_2$.

Now suppose instead that $Q_1 \in \mathcal{P}_1(\F_2^n)$ and $c_0,c_1,c_2,c_3 \in \F_2$, and consider the average
\begin{equation}\label{efn}
 \E_{x \in \F_2^n} (-1)^{S_4(x) + c_3 S_3(x) + c_2 S_2(x) + c_1 S_1(x) + c_0 S_0 - Q_1} .
\end{equation}
Then we can write 
$$ Q_1(x) = \sum_{i \in E} x_i + Q_0(x)$$
for some $Q_0 \in \mathcal{P}_0(\F_2^n)$ and some set $E \subset \{1,\ldots,n\}$.  We can thus find a set $I \subseteq \{1,\ldots,n\}$ of size $m := \lfloor \frac{n}{2} \rfloor$ which either lies in $E$, or is disjoint from $E$.  By permuting the coefficients we can write $I = \{1,\ldots,m\}$.  Then by freezing the coefficients $y := (x_{m+1},\ldots,x_n) \in \F_2^{n-m}$, we see that we can write \eqref{efn} as an average of expressions of the form
$$ \E_{x \in \F_2^m}(-1)^{S_4(x) + c_{3,y} S_3(x) + c_{2,y} S_2(x) + c_{1,y} S_1(x) + c_{0,y} S_0 - Q_{0,y}} $$
for some $c_{0,y},\ldots,c_{3,y} \in \F_2$ and $Q_{0,y} \in \mathcal{P}_1(\F_2^m)$.  Applying \eqref{base} and the triangle inequality we thus conclude that
\begin{equation}\label{efn-2}
\E_{x \in\F_2^n} (-1)^{S_4(x) + c_3 S_3(x) + c_2 S_2(x) + c_1 S_1(x) + c_0 S_0 - Q_1(x)} ) = O(2^{-\Omega(n)}) .
\end{equation}

Now suppose instead that $Q_2 \in \mathcal{P}_2(\F_2^n)$ and $c_0,c_1,c_2,c_3 \in \F_2$, and consider the average
$$
 \E_{x \in \F_2^n} (-1)^{S_4(x) + c_3 S_3(x) + c_2 S_2(x) + c_1 S_1(x) + c_0 S_0 - Q_2(x)} .
$$
Then we can write
$$ Q_2(x) = \sum_{\{i,j\} \in E(\Gamma)} x_i x_j + Q_1(x)$$
for some $Q_1 \in \mathcal{P}_1(\F_2^n)$ and some graph $\Gamma$ on vertex set $[n]$ .  By Ramsey's theorem (see e.g. \cite[Section 4.2]{graham}), we can find a set $I \subseteq [n]$ of size $m = \Omega(\log n)$ such that the complete graph on vertex set $I$ either lies completely inside $E$, or is disjoint from $E$.  We can then repeat the above freezing argument (using \eqref{efn-2} instead of \eqref{base}) and conclude that
$$  \E_{x \in \F_2^n}(-1)^{S_4(x) + c_3 S_3(x) + c_2 S_2(x) + c_1 S_1(x) + c_0 S_0 - Q_2(x)}  = O( 2^{-\Omega(m)} ) = O( n^{-\Omega(1)} ).$$

Finally, suppose $Q_3 \in \mathcal{P}_1(\F_2^n)$ and $c_0,c_1,c_2,c_3 \in \F_2$, and consider the average
$$
 \E_{x \in \F_2^n} (-1)^{S_4(x) + c_3 S_3(x) + c_2 S_2(x) + c_1 S_1(x) + c_0 S_0 - Q_3(x)} .
$$
Then we can write
$$ Q_3(x) = \sum_{\{i,j,k\} \in E(\Gamma)} x_i x_j x_k + Q_2(x)$$
for some $Q_2 \in \mathcal{P}_2(\F_2^n)$ and some $3$-uniform hypergraph $\Gamma$ on vertex set $[n]$. Applying the bounds of Erd\H{o}s and Rado for the hypergraph Ramsey theorem (see e.g. \cite[Section 4.7]{graham}) we can find
a set $I \subset [n]$ of size $m = \Omega(\log \log n)$ such that the complete $3$-uniform hypergraph on $I$ either lies completely inside $E$ or is disjoint from $E$.  Using the freezing argument one last time, we obtain
\begin{equation}\label{u4-weak}
\E_{x \in \F_2^n}(-1)^{S_4(x) + c_3 S_3(x) + c_2 S_2(x) + c_1 S_1(x) + c_0 S_0 - Q_3(x)}  = O( m^{-\Omega(1)} ) = O( (\log \log n)^{-\Omega(1)} ).
\end{equation}
This is a bound of the form claimed in \eqref{u4b} of Theorem \ref{main-1}, except there is an extra logarithm. To remove it, we run the two Ramsey-theoretic arguments in parallel, by using the following variant of the Erd\H{o}s-Rado bound.

\begin{lemma}[Simultaneous Ramsey theorem]  Let $E_2 \subseteq \binom{[n]}{2}$ and $E_3 \subseteq \binom{[n]}{3}$ be a graph and $3$-uniform hypergraph respectively.  Then there exists a set $I \subset [n]$ of size $m = \Omega(\log \log n)$ such that for each $j=2,3$, the set $\binom{I}{j}$ either lies completely inside $E_j$ or is disjoint from $E_j$.
\end{lemma}

\begin{proof}  We generate some vertices $x_1,\ldots,x_l$ by the following algorithm:
\begin{itemize}
\item Step 0.  Initialise $l=0$ and $J := [n]$.
\item Step 1.  By the pigeonhole principle, there exists $J' \subseteq J$ with $|J'| \gg 2^{-O(l^2)} |J|$ such that for any $i,j \in [l]$ and $x \in J'$, the truth value of the statements $\{ x_i, x\} \in E_2$ or $\{x_i,x_j,x\} \in E_3$ are independent of $x$.  Fix this $J'$.
\item Step 2.  Set $x_{l+1} := \min(J')$, replace $J$ by $J' \backslash \{x_{l+1}\}$, and increment $l$ to $l+1$.  If $J'$ is non-empty then return to Step 1; otherwise STOP.
\end{itemize}
One easily verifies that this algorithm terminates in $k = \Omega( \log^{1/3} n )$ steps to obtain a sequence $1 \leq x_1 \leq \ldots\leq x_l \leq n$ with the property that for any $1 \leq i < j \leq l$, the truth value of $\{ x_i, x_j \} \in E_2$ is independent of $j$, and for any $1 \leq i < j < k \leq l$, the truth value of $\{ x_i, x_j, x_k \} \in E_3$ is independent of $k$.  By an appeal to Ramsey's theorem for graphs one can then find a set $I \subset \{x_1,\ldots,x_k\}$ with $|I| \gg \log k \gg \log \log n$ with the desired properties.
\end{proof}

Note that by applying Ramsey's theorem for graphs and $3$-uniform hypergraphs sequentially, one would only get $m = \Omega(\log\log\log n)$ here.  The reader can easily verify that the logarithmic saving in this lemma propagates through the previous arguments to improve \eqref{u4-weak} to \eqref{u4b}.
\endproof

\section{General degrees and characteristics}\label{gen-sec}

It is natural to wonder for which $\F$ and $d$ the symmetric polynomials $S_d$ on $\F^n$ provide counterexamples to Conjecture \ref{ig}, the inverse conjecture for the $U^d$-norm. We do not have a complete answer to this question, but we give some partial results in this direction here.  For a more in-depth treatment of these issues, we refer the reader to the recent preprint \cite{lms}.

We begin with a general result that shows that $\| e_\F(S_d) \|_{U^d}$ is large whenever $d > |\F|$. This result (and in fact a generalisation of it which establishes the largeness of $\| e_{\F}(S_d) \|_{U^{d-p+2}}$ for $d \geq 2p$, where $p = |\F|$) was shown to us by the authors of \cite{lms} before we wrote this section. The following argument is a slight variant of theirs which, we believe, is worth having in the literature.

\begin{theorem}[Lower bound on Gowers norm]\label{lowgow}  Let $\F$ be a finite field, let $n \geq 1$, and let $d > |\F|$.  Let $S_d$ be the symmetric polynomial on $\F^n$, and let $f := e_\F(S_d)$.  Then $\|f\|_{U^d} \gg_{\F} 1$.
\end{theorem}

\begin{proof}
For this, we must find some analogue of the computations earlier in the section and, in particular, the identity \eqref{qident}.  For this we need some more notation.  Let $\Pi_n$ denote the collection of all partitions $\pi = \{ C_1,\ldots,C_m\}$ of $[n]$ into disjoint sets $[n] = C_1 \cup \ldots \cup C_m$. For any partition $\pi = \{ C_1,\dots,C_m \} \in \Pi_n$, we associate the multilinear form $R_\pi: \F^n \times \ldots \times \F^n \to \F$ by
\[ R_{\pi}(h^{(1)},\dots,h^{(d)}) := \prod_{k=1}^m \sum_{j = 1}^n \prod_{i \in C_k} h^{(i)}_j.\]
Thus for example if $\pi$ is the partition of $[3]$ into $\{1,2\}$ and $\{3\}$ then we have
\[ R_{\pi}(h^{(1)}, h^{(2)},h^{(3)}) = (h^{(1)}_1h^{(2)}_1 + \dots + h^{(1)}_nh^{(2)}_n)(h^{(3)}_1 + \dots + h^{(3)}_n).\]
We define the \emph{M\"obius function} $\mu(\pi)$ of $\mu$ at $\pi$ by the formula
\begin{equation}\label{mupi}
\mu(\pi) := \prod_k (-1)^{|C_k|} (|C_k| - 1)!.
\end{equation}

We place a partial ordering on partitions $\pi$ by declaring $\pi' \preceq \pi$ if every set in $\pi'$ is contained in some set in $\pi$.  This has a minimal element $\pi_{\min} := \{ \{1\},\ldots,\{n\}\}$.  The M\"obius function can be shown\footnote{See for instance the series of exercises \cite[p. 103]{birkhoff}, or \cite[Lemma 4.1]{crt}.} to obey the M\"obius inversion identities $\mu(\pi_{\min}) = 1$ and $\sum_{\pi' \preceq \pi} \mu(\pi') = 0$ if $\pi \neq \pi_{\min}$. 

As a consequence we obtain the following variant of \eqref{qident}, which follows from \cite[Proposition 2.7]{lms}.

\begin{lemma}[Derivative of symmetric function]\label{symderiv}  For any $d \geq 1$ and $h^{(1)},\ldots,h^{(d)},x \in \F^n$, we have
\begin{equation}\label{qident-gen} D_{h^{(1)}}\ldots D_{h^{(d)}}S_d(x) = \sum_{\pi} \mu(\pi) R_\pi(h^{(1)},\ldots,h^{(d)}).\end{equation}
\end{lemma}

\begin{proof}  
Each $R_{\pi}$ may be expanded as a sum
\begin{equation}\label{eq227} R_{\pi}(h^{(1)},\dots,h^{(d)}) = \sum_{\pi \preceq\tau(i_1,\dots,i_n)} h^{(1)}_{i_1} \dots h^{(n)}_{i_n},\end{equation}
where $\tau(i_1,\dots,i_n)$ is the partition on $[n]$ induced by the indices $i_1,\dots,i_n$, two elements $s,t$ being placed in the same element of this partition if and only if $i_s = i_t$. 

On the other hand, from proof of Theorem \ref{main-1} we have
\begin{equation}\label{eq228}
\begin{split}
D_{h^{(1)},\dots,h^{(d)}}S_d(x) &= \sum_{\substack{1 \leq i_1,\dots,i_d \leq n \\ i_1,\dots,i_d \; \mbox{\scriptsize distinct}}} h^{(1)}_{i_1} \dots h^{(d)}_{i_d}\\
&= \sum_{\tau(i_1,\dots,i_n) = \pi_{\min}} h^{(1)}_{i_1} \dots h^{(n)}_{i_n}.
\end{split}
\end{equation} The claim now follows from the M\"obius inversion formula.
\end{proof}

To apply the identity \eqref{qident-gen}, we let $V \subseteq \F^n$ be the variety
$$ V := \{ x \in \F_2^n: S_1(x) = S_2(x) = \ldots = S_p(x) = 0 \},$$ where $p = |\F|$ (later on we will specialize to the case $p = 2$).
We claim the identity
\begin{equation}\label{defv}
 \Delta_{h^{(1)}} \ldots \Delta_{h^{(d)}}( f 1_V )(x) = \Delta_{h^{(1)}} \ldots \Delta_{h^{(d)}}( 1_V )(x) 
\end{equation}
for $x, h^{(1)},\ldots,h^{(d)} \in \F^n$.  To prove \eqref{defv}, it suffices to show that
$$ D_{h^{(1)}}\ldots D_{h^{(d)}}S_d(x) = 0$$
whenever $x, h^{(1)},\ldots,h^{(d)} \in \F^n$ are such that the cube $\{ x + \omega_1 h^{(1)} + \ldots + \omega_d h^{(d)}: \omega_1,\ldots,\omega_d \in \{0,1\} \}$ lies in $V$.  But if $x,h^{(1)},\dots,h^{(d)}$ are such elements then, by definition of $V$ and differentiation, we have
\begin{equation}\label{dhj}
D_{h^{(i_1)}} \ldots D_{h^{(i_j)}} S_j(x) = 0
\end{equation}
for all $j \in \{1,\dots,p\}$ and distinct $i_1,\ldots,i_j \in \{1,\ldots,d\}$. Note from \eqref{mupi} that the M\"obius function $\mu(\pi_{\triv,j})$ is invertible in $\F$ for all $1 \leq j \leq p$, where $\pi_{\triv,j}$ is the trivial partition $\{\{1,\dots,j\}\}$ of $[j]$. By expanding the left-hand side of \eqref{dhj} using the inversion formula \eqref{qident-gen}, we conclude recursively that
$$ R_{\pi_{\triv,j}}( h^{(i_1)},\ldots,h^{(i_j)} ) = 0$$
for all $1 \leq j \leq p$ and distinct $i_1,\ldots,i_j \in \{1,\ldots,d\}$.  This implies that
$$ R_{\pi}( h^{(1)},\ldots,h^{(d)} ) = 0$$
whenever all sets in $\pi$ have cardinality at most $p$.  On the other hand, if any set in $\pi$ has cardinality greater than $p$, we see from \eqref{mupi} that $\mu(\pi)$ vanishes in $\F$.  The claim \eqref{defv} now follows from one last application of \eqref{qident-gen}.

Using \eqref{defv} and Definition \ref{gd}, we conclude that
$$ \| f 1_V \|_{U^d} = \|1_V\|_{U^d}.$$
But by monotonicity of Gowers norms (see e.g. \cite[Chapter 11]{tao-vu}) we have
$$ \|1_V\|_{U^d} \geq \|1_V\|_{U^1} =  |V|/|\F^n|.$$
By applying Lemma \ref{mpr} we have $|V|/|\F^n| \gg_{\F} 1$, and so
$$ \| f 1_V \|_{U^d} \gg_{\F} 1.$$
On the other hand, we have the Fourier expansion
$$ 1_V = \E_{\xi \in \F^p} e_\F( \xi_1 S_1 + \ldots + \xi_p S_p ).$$
Using the triangle inequality for Gowers norms (see e.g. \cite[Lemma 3.9]{gowers} or \cite[Chapter 11]{tao-vu}) we conclude that
$$ \| f e_\F( \xi_1 S_1 + \ldots + \xi_p S_p ) \|_{U^d} \gg_{\F} 1$$
for some $\xi_1,\ldots,\xi_p \in \F$.  Theorem \ref{lowgow} now follows from \eqref{polyphase} and the hypothesis that $d > p$.
\end{proof}

As a consequence of the above theorem, we can completely characterise the behaviour of $(-1)^{S_d}$ in the characteristic $2$ case.

\begin{theorem}[Gowers norm behaviour of $S_d$ over $\F_2$]\label{sdgow}  Let $n \geq 1$ and $d \geq 1$ be integers, let $\F = \F_2$, and let $f := (-1)^{S_d}$ where $S_d$ is the $d^{\th}$ elementary symmetric function on $\F_2^n$. 
\begin{itemize}
\item If $d=1,2$, then $\|f\|_{U^d}, \|f\|_{u^d} = o(1)$.
\item If $d$ is not a power of $2$, then $\rank_{d-1}(S_d) \leq 2$ and $\|f\|_{U^d} \geq \|f\|_{u^d} \geq \frac{1}{4}$.
\item If $d$ is a power of $2$ which is at least $4$, then $\|f\|_{U^d} \gg 1$ and $\|f\|_{u^d} = o_d(1)$, where $o_d(1)$ goes to zero as $n \to \infty$ for fixed $d$.  \textup{(}In particular, Conjecture \ref{ig} fails for the $U^d$-norm on $\F_2^n$ for these values of $d$.\textup{)}
\end{itemize}
\end{theorem}

\begin{proof}  The cases $d=1,2$ can be computed by hand (using Lemma \ref{gauss} for the $d=2$ case).  
If $d$ is not a power of $2$, then from Lucas' theorem \eqref{lucas} we can express $S_d$ as a product $S_{d_1} S_{d_2}$ for some $d_1,d_2$ with $0 < d_1,d_2 < d$ and $d=d_1+d_2$, which gives the desired bound on $\rank_{d-1}(S_d)$. By Fourier analysis in $\F^{k+1}_2$ we may therefore write
\[ (-1)^{S_d} = \frac{1}{4}( 1 + (-1)^{S_{d_1}} + (-1)^{S_{d_2}} + (-1)^{S_{d_1}+S_{d_2}} ). \] Thus $(-1)^{S_d}$ must have an inner product of at least $\frac{1}{4}$ with at least one polynomial phase of degree strictly less than $d$, which gives the lower bound on $\|f\|_{u^d}$ in this case.  The lower bound on $\|f\|_{U^d}$ then follows from \eqref{uU}.

When $d$ is a power of $2$, one verifies (as in the proof of Theorem \ref{main-1}) that $S_d(x) = 1$ precisely when $x$ is equal to $d,\ldots,2d-1 \md{2d}$, whereas $S_{d'}$ for $d' < d$ is periodic with period dividing $d$.  Using multisection of series as before, we can conclude an analogue of \eqref{base} for $S_d$ instead of $S_4$, and by repeating the Ramsey arguments one obtains the desired bound $\|f\|_{u^d} = o_d(1)$.  Finally, the lower bound on $\|f\|_{U^d}$ follows from Theorem \ref{lowgow}.  This establishes all the claims of the theorem.
\end{proof}

\begin{remark}  When $\F=\F_2$ and $d$ is a power of two, the above theorem shows that $(-1)^{S_d}$ does not correlate strongly with any polynomial phase in $\F_2^n$ of order $d-1$ or less.  However, the argument we used to prove this showed that $S_d$ was still \emph{locally polynomial} of degree $d-1$ on the subvariety $V := \{ x \in \F_2^n: S_1(x)=S_2(x) = 0 \}$, in the sense of \cite{gt:inverse-u3}.  This raises the possibility that Conjecture \ref{ig} may be salvaged by working with \emph{locally} polynomial phases instead of global ones; in fact this formulation of the conjecture was already implicit in \cite[Section 13]{gt:inverse-u3}. 
\end{remark}

\end{document}